\documentclass{amsart}
\usepackage{graphicx} % Required for inserting images

\usepackage{amssymb}
\usepackage{amscd}
\usepackage[margin=1in]{geometry}
\usepackage{xcolor}

\newtheorem{theorem}{Theorem}
\newtheorem{corollary}[theorem]{Corollary}
\newtheorem{proposition}[theorem]{Proposition}
\newtheorem{lemma}[theorem]{Lemma}
\newtheorem{conjecture}[theorem]{Conjecture}
\newtheorem{remark}[theorem]{Remark}
\newtheorem{definition}[theorem]{Definition}

\newcommand{\grad}{\nabla}
\newcommand{\lap}{\Delta}
\newcommand{\R}{\mathbb{R}}
\newcommand{\f}{\colon}
\newcommand{\bd}{\partial}
\newcommand{\inv}{^{-1}}
\newcommand{\eps}{\varepsilon}

\newcommand{\area}{\operatorname{Area}}
\newcommand{\RP}{\mathbb{RP}}
\newcommand{\la}{\langle}
\newcommand{\ra}{\rangle}
\newcommand{\ric}{\operatorname{Ric}}
\renewcommand{\div}{\operatorname{div}}

\title{Monotone Quantities for $p$-Harmonic functions and the Sharp $p$-Penrose inequality}
\author{Liam Mazurowski}
\address{Department of Mathematics \\ Cornell University \\ Ithaca, NY, 14853}
\email{lmm334@cornell.edu}

\author{Xuan Yao}
\address{Department of Mathematics \\ Cornell University \\ Ithaca, NY, 14853}
\email{xy346@cornell.edu}

\date{\today}

\begin{document}

\maketitle

\begin{abstract} Consider a complete asymptotically flat 3-manifold $M$ with non-negative scalar curvature and non-empty minimal boundary $\Sigma$. Fix a number $1 < p < 3$. We derive monotone quantities for $p$-harmonic  functions on $M$ which become constant on Schwarzschild. These monotonicity formulas imply a sharp mass-capacity estimate relating the ADM mass of $M$ with the $p$-capacity of $\Sigma$ in $M$, which was first proved by Xiao using weak inverse mean curvature flow. 
\end{abstract}

\section{Introduction}

Asymptotically flat Riemannian 3-manifolds arise naturally in the study of general relativity, where they model space-like slices 
with vanishing second fundamental form inside of 4-dimensional space-time. Let $(M^3,g)$ be an asymptotically flat Riemannian 3-manifold. The ADM mass \cite{arnowitt1961coordinate} of $M$  is a number $m_{\text{ADM}}$ associated to $M$ which, roughly speaking, measures how quickly $M$ flattens out at infinity. Motivated by physical considerations, it was conjectured that the ADM mass of any asymptotically flat manifold with non-negative scalar curvature should be non-negative. Moreover, the mass should equal zero if and only if $M$ is isometric to Euclidean space. This so-called Positive Mass Theorem was first proven by Schoen and Yau \cite{schoen1979proof} using techniques from the theory of minimal surfaces. Later Witten \cite{witten1981new} gave another proof using spinor methods.  

A related conjecture of Penrose \cite{penrose1973naked} states that if $M$ has non-negative scalar curvature and minimal boundary $\Sigma$ then the ADM mass should satisfy 
\[
m_{\text{ADM}} \ge \sqrt{\frac{\area(\Sigma)}{16\pi}}.
\]
Moreover, equality should hold only on the spatial Schwarzschild manifolds. 
Geroch \cite{geroch1973energy} observed that if $M$ has non-negative scalar curvature and $\Sigma$ is connected then the Hawking mass increases along smooth inverse mean curvature flow starting from $\Sigma$. This leads to a proof of the Riemannian Penrose Inequality in the case where the inverse mean curvature flow starting from $\Sigma$ is smooth for all time. However, inverse mean curvature flow starting from $\Sigma$ may develop singularities in general. Huisken and Ilmanen \cite{huisken2001inverse} developed a theory of weak inverse mean curvature flow which ``jumps'' past singularities. Moreover, they showed that Hawking mass is still monotone along weak inverse mean curvature flow. Thus Huisken and Ilmanen were able to prove the Riemannian Penrose Inequality in the case where $\Sigma$ is connected.

Bray \cite{bray2001proof} gave another proof of the Riemannian Penrose Inequality by using the Positive Mass Theorem. As part of the proof, Bray obtained a mass-capacity inequality relating the ADM mass of $M$ to the harmonic capacity of the boundary $\Sigma$.  More precisely, let $u$ be the harmonic function on $M$ which is 1 on $\Sigma$ and goes to 0 at infinity. Then Bray showed that 
\[
m_{\text{ADM}} \ge \frac{1}{4\pi}\int_{\Sigma} \vert \grad u\vert\, da,
\]
with equality if and only if $(M,g)$ is isometric to Schwarzschild with $\Sigma$ as the horizon.

Recently, a number of authors have discovered connections between curvature and the level sets of harmonic and, more generally, $p$-harmonic functions. Fundamental work of Colding \cite{colding2012new} and Colding-Minicozzi \cite{colding2014ricci} revealed that the quantity 
\[
A_\beta(r) = r^{1-n} \int_{\{v=r\}} \vert \grad v\vert^{1+\beta}\, da
\]
is decreasing, where $v^{2-n}$ is a Green's function on a nonparabolic manifold $M^n$ with non-negative Ricci curvature and $\beta \ge \frac{n-2}{n-1}$ is a fixed constant. 
Colding used this to prove a sharp gradient estimate for the Green's function on any such manifold $M$. A number of further applications have been given in manifolds with non-negative Ricci curvature. For instance, Agostiniani-Mazzieri \cite{agostiniani2020monotonicity}  employed the monotonicity of a related quantity to prove a quantitative Willmore-type inequality in $\R^n$, and Agostiniani-Fogagnolo-Mazzieri \cite{agostiniani2020sharp} generalized this to prove a Willmore-type inequality in complete manifolds $M$ with non-negative Ricci curvature and Euclidean volume growth. Later, Agostiniani-Fogagnolo-Mazzieri \cite{agostiniani2022minkowski} derived similar monotonicity formulas  along the level sets of a $p$-harmonic function in $\R^n$. As an application, they obtained an $L^p$ version of the classical Minkowski inequality. Finally, Benatti-Fogagnolo-Mazzieri \cite{benatti2021minkowski} proved a  Minkowski inequality for complete manifolds $M$ with non-negative Ricci curvature and Euclidean volume growth.

The level sets of harmonic and $p$-harmonic functions can also be used to study 3-manifolds with non-negative scalar curvature. This connection was  observed by Stern \cite{stern2022scalar}, who re-arranged the Bochner formula for harmonic functions to take advantage of a scalar curvature term. Stern was thereby able to give a new proof that the torus $T^3$ admits no metric of positive scalar curvature. Since then, harmonic and $p$-harmonic functions have found a number of applications in the study of asymptotically flat manifolds with non-negative scalar curvature. We note the following results, listed in chronological order. Here $(M^3,g)$ is an asymptotically flat manifold with non-negative scalar curvature. 
\begin{itemize} 
\item Bray, Kazaras, Khuri, and Stern \cite{bray2022harmonic} considered a harmonic functions $u$ asymptotic to a coordinate function on $M$ at infinity. They proved that
\[
m_{ADM}(M) \ge \int_M \frac{\vert \grad^2 u\vert^2}{\vert \grad u\vert} + \frac{R}{\vert \grad u\vert}\, dv,
\]
and thus obtained a new proof of the Positive Mass Theorem. 
\item Munteanu and Wang \cite{munteanu2021comparison} discovered a monotone quantity 
\[
\mathcal G(t) = -4\pi t + \frac{1}{t}\int_{\{u=t\}} \vert \grad u\vert^2\, da
\]
along the level sets of a Green's function $u$ in $M$. As an application, they derived a sharp area comparison for the level sets of the Green's function.
\item Agostiniani, Mazzieri, and Oronzio \cite{agostiniani2021green} discovered a monotone quantity 
\[
F(t) = 4\pi t - t^2 \int_{\{u=1-\frac{1}{t}\}} \vert \grad u\vert H\, da + t^3 \int_{\{u=1-\frac{1}{t}\}} \vert \grad u\vert^2\, da
\]
along the level sets of a Green's function $u$ in $M$. They further showed that the monotonicity of $F$ leads to a new proof of the Positive Mass Theorem. 
\item Chan, Chu, Lee, and Tsang \cite{chan2022monotonicity} generalized the monotonicity of $\mathcal G$ to the level sets of a $p$-harmonic Green's function.
\item Agostiniani, Mantegazza, Mazzieri, and Oronzio \cite{agostiniani2022riemannian} generalized the monotonicity of $F$ to the level sets of $p$-harmonic functions.  As an application, they obtained a new proof of the Riemannian Penrose Inequality (without rigidity).
\item Miao \cite{miao2022mass} discovered three monotone quantities 
\begin{gather*}
\mathcal A(t) = \frac{1}{1-t}\left(8\pi - \frac{1}{1-t}\int_{\{u=t\}} H\vert \grad u\vert\, da\right),\\
\mathcal B(t) = \frac{1}{1-t}\left(4\pi - \frac{1}{(1-t)^2}\int_{\{u=t\}} \vert \grad u\vert^2\, da\right),  \\
\mathcal D(t) = 4\pi(1-t) + \int_{\{u=t\}} H\vert \grad u\vert - \frac{3}{1-t}\int_{\{u=t\}} \vert \grad u\vert^2\, da,
\end{gather*} 
along the level sets of a harmonic function $u$ on $M$, and thereby obtained several new proofs of the Positive Mass Theorem. 
\item Hirsch, Miao, and Tam \cite{hirsch2022monotone} generalized the monotonicity of $\mathcal A$,  $\mathcal B$, and $\mathcal D$ to the level sets of $p$-harmonic functions. They also gave a unified derivation of the monotonicity of $\mathcal A$, $\mathcal B$, $\mathcal D$, $F$, and $\mathcal G$ for general $p$, including the case $p=1$ where all of these quantities degenerate to the Hawking mass along inverse mean curvature flow. 
\item Dong and Song \cite{dong2023stability} used harmonic functions to prove the stability of the Positive Mass Theorem. 
\end{itemize}
The previously mentioned monotone quantities all have Euclidean space as their rigidity case. In particular, they are strictly increasing (or strictly decreasing) on Schwarzschild with positive mass. We note that Miao \cite{miao2022mass} and Oronzio \cite{oronzio2022adm} have obtained monotone quantities with Schwarzschild rigidity in the harmonic case $p=2$. Miao's quantities are related to $\mathcal A$ and $\mathcal B$ while Oronzio's are related to $F$. The authors \cite{mazurowski2023yamabe} recently used these monotone quantities to give a new proof of a result of Bray and Neves \cite{bray2004classification} on the Yamabe invariant of $\RP^3$.

 In this paper, we derive monotone quantities for $p$-harmonic functions with Schwarzschild rigidity for all $1<p<3$. In order to state the main results , we first need to introduce some notation. To begin, we recall the definition of an asymptotically flat manifold. 

\begin{definition}
Suppose $(M,g)$ is a smooth connected $n$-dimensional Riemannian manifold, such that there exists a compact set $\mathcal{K}$, $M\setminus\mathcal{K}=\sqcup_{k=1}^{k_0}M_{\text{end}}^k$, where for each $k$, $M_{\text{end}}^k$ is diffeomorphic to the complement of a ball $B$ in $\mathbb{R}^n$, and in each end, there is a coordinate system satisfying
\begin{align*}
    |\partial^l(g_{ij}-\delta_{ij})(x)|=O(|x|^{-q-l}),
\end{align*}
for some $q>\frac{1}{2}$, $l=0,1,2$. Then we call $(M^n,g)$ an asymptotically flat manifold.
\end{definition}

In the rest of the paper, $(M^3,g)$ will always denote a complete, asymptotically flat manifold with non-negative scalar curvature and non-empty, minimal boundary $\Sigma$. We assume that $M$ has only one end. Fix a number  $1 < p < 3$ and let $u$ be the solution to 
\begin{equation}
\label{u-pde}
\begin{cases}
\lap_p u = 0, &\text{on } M\\
u = 1, &\text{on } \Sigma,\\
u\to 0, &\text{at infinity}.
\end{cases}
\end{equation}
Let 
\begin{equation}
w = (1-p)\log u
\end{equation} and define 
\begin{equation}
\label{W-definition}
W(t) = \int_{\{w=t\}} \vert \grad w\vert^2\, da. 
\end{equation}
As first noted by Moser \cite{moser2007inverse}, when $p$ is close to 1, the function $w$ approximates inverse mean curvature flow starting from $\Sigma$. Moreover, in this case, the quantity $W$ approximates the integral of $H^2$ over a level set in inverse mean curvature flow. 

\begin{remark}
The above notation differs from that commonly used in the literature. In \cite{agostiniani2022riemannian} and \cite{hirsch2022monotone}, for example, $u$ denotes the $p$-harmonic function which is 0 on $\Sigma$ and goes to 1 at infinity, and all other quantities are defined directly in terms of $u$. We have chosen to use the above notation because we feel it better elucidates the relationship between inverse mean curvature flow and the level sets of $p$-harmonic functions. 
\end{remark}

As our main result, we find monotone quantities along the level sets of the $p$-harmonic function $u$. These quantities are constant if and only if  $M$ is Schwarzschild. 
 As noted above, for $1<p<3$, previous authors already found monotone quantities associated to $p$-harmonic functions.However, while these quantities are constant on Euclidean space, they are all strictly increasing when $M$ is Schwarzschild. In the special case $p=2$, Miao \cite{miao2022mass} and Oronzio \cite{oronzio2022adm} found monotone quantites associated to harmonic functions which are constant on Schwarzschild. To the authors' knowledge, our monotone quantities are the first to have Schwarzschild rigidity in the case $p\in (1,2)\cup(2,3)$.
Our monotone quantities are of the form 
\[
Q(t) = 4\pi(3-p)^2 f(t) + g(t)W(t) + (p-1)(3-p)h(t)\frac{dW}{dt}(t)
\]
where $f$, $g$, and $h$ solve a certain system of ODEs. We note in particular two special choices for the functions $f$, $g$, and $h$ which seem  useful for applications. The first choice $f_*$, $g_*$, $h_*$ gives a quantity $Q_*$ which goes to 0 as $t\to \infty$. The second choice $f^*$, $g^*$, $h^*$ gives a quantity $Q^*$ which detects the mass as $t\to \infty$. These monotone quantities agree with those found by Oronzio \cite{oronzio2022adm} in the case $p=2$. 

\begin{theorem}
\label{theorem:decaying}
Fix a number $1<p<3$.  There exist functions $f_*,g_*,h_*\f [0,\infty)\to \R$ with the following properties. Let $(M^3,g)$ be a complete asymptotically flat manifold with non-negative scalar curvature and non-empty minimal boundary $\Sigma$. Assume that $H_2(M,\Sigma) = 0$ and let $u$, $w$, and $W$ be as above.  For each regular value $t$ of $w$ define  
\[
Q_*(t) = 4\pi(3-p)^2 f_*(t) + g_*(t)W(t) + (p-1)(3-p)h_*(t)\frac{dW}{dt}(t).
\]
Then $Q_*(s) \le Q_*(t)$ for any regular values $s \le t$ of $w$. Moreover, $Q_*$ is constant if and only if $(M^3,g)$ is isometric to Schwarzschild with horizon boundary. Finally one has $f_*(0) < 0$ and $g_*(0)+2(3-p)h_*(0) > 0$ and $\lim_{t\to\infty} Q_*(t) =0$.
\end{theorem}

\begin{theorem}
\label{theorem:growing}
Fix a number $1<p<3$.  There exist functions $f^*,g^*,h^*\f [0,\infty)\to \R$ with the following properties. Let $(M^3,g)$ be a complete asymptotically flat manifold with non-negative scalar curvature and non-empty minimal boundary $\Sigma$. Assume that $H_2(M,\Sigma) = 0$ and let $u$, $w$, and $W$ be as above.  For each regular value $t$ of $w$ define  
\[
Q^*(t) = 4\pi(3-p)^2 f^*(t) + g^*(t)W(t) + (p-1)(3-p)h^*(t)\frac{dW}{dt}(t).
\]
Then $Q^*(s) \le Q^*(t)$ for any regular values $s \le t$ of $w$. Moreover, $Q^*$ is constant if and only if $(M^3,g)$ is isometric to Schwarzschild with horizon boundary. Finally one has 
\[
\lim_{t\to \infty} Q^*(t) \le 8\pi(3-p)^2\left(\frac{p-1}{3-p}\right)^{-\frac{p-1}{3-p}} \left(\frac{K_p}{C_p}\right)^{\frac{1}{3-p}} m_{\operatorname{ADM}}-4\pi\left((3-p)^2+4\right)
\]
and $g^*(0) + 2(3-p)h^*(0) < 0$.
\end{theorem}

By comparing the values of $Q_*$ and $Q^*$ at $0$ with their limits as $t\to \infty$, we obtain a new proof of the following sharp $p$-Penrose inequality. This inequality interpolates between the Riemannian Penrose Inequality when $p\to 1$ and Bray's mass-capacity inequality for harmonic functions when $p\to 2$.

\begin{definition}
The $p$-capacity of $\Sigma$ in $M$ is the quantity
\begin{equation}
\label{p-capacity-equation}
C_p = \int_{\Sigma} \vert \grad u\vert^{p-1}\, da.
\end{equation}
Let $K_p$ denote the $p$-capacity of the horizon in mass 2 Schwarzschild.
\end{definition}

\begin{corollary}[Sharp $p$-Penrose inequality]
\label{main-theorem}
Fix a number $1 < p < 3$. Assume that $(M^3,g)$ is a complete, asymptotically flat manifold with non-negative scalar curvature and non-empty minimal boundary $\Sigma$. Assume that $H_2(M,\Sigma) = 0$. Then 
\[
m_{\operatorname{ADM}} \ge 2\left(\frac{C_p}{K_p}\right)^{\frac{1}{3-p}}.
\]
Moreover, equality holds if and only if $(M,g)$ is isometric to Schwarzschild with $\Sigma$ as the horizon. 
\end{corollary}

The sharp $p$-Penrose inequality was first proved by Xiao \cite{xiao2016p} using weak inverse mean curvature flow.  Our proof is based purely on non-linear potential theory. This has the advantage of avoiding the technical work needed to establish the existence of Huisken and Ilmanen's weak inverse mean curvature flow. It is also perhaps psychologically satisfying that our proof of the mass to $p$-capacity inequality uses only $p$-harmonic functions. For this application, it is essential that our monotone quantities are constant on Schwarzschild. We note that Agostiniani-Mantegazza-Mazzieri-Oronzio \cite{agostiniani2022riemannian} and Benatti-Fogagnolo-Mazzieri \cite{benatti2023nonlinear} also previously proved some mass-capacity inequalities for $p$-harmonic functions. However, their inequalities do not yield the sharp $p$-Penrose inequality, as equality does not hold on Schwarzschild. We would also like to mention that, after our preprint appeared on the arXiv, similar results were independently obtained in a preprint by Xia, Yin, and Zhou. 

\subsection{Outline of proof}

Assume that $(M,g)$ is a complete, asymptotically flat manifold with non-negative scalar curvature and non-empty minimal boundary $\Sigma$. Assume for simplicity that $\Sigma$ is topologically a 2-sphere and that the solution $u$ to (\ref{u-pde}) has no critical points. In Section \ref{section:W-inequality}, we show that $W$ satisfies the following 2nd order differential inequality: 
\begin{equation}
\label{1}
(p-1)(3-p)\frac{d^2W}{dt^2} \ge W - 4\pi(3-p)^2 + 2(2-p)\frac{dW}{dt} + \frac{(p-1)(5-p)}{4} \frac{\left(\frac{dW}{dt}\right)^2}{W}. 
\end{equation}
Equality holds when $M$ is Schwarzschild. We note in particular that the non-negative term involving $(\frac{dW}{dt})^2$ is identically zero in Euclidean space but is positive when $M$ is Schwarzschild. Dropping this term gives the simpler inequality 
\begin{equation}
\label{simpler}
(p-1)(3-p)\frac{d^2W}{dt^2} \ge W - 4\pi(3-p)^2 + 2(2-p)\frac{dW}{dt},
\end{equation}
where equality holds in Euclidean space but not on Schwarzschild. 
Some version of this inequality was implicitly present in several previous works, although it was not explicitly written down to the authors' knowledge. Inequality (\ref{simpler}) can be used to systematically derive the monotonicity of $\mathcal A$, $\mathcal B$, $\mathcal D$, $F$, and $\mathcal G$, at least when no critical points are present. However, inequality (\ref{simpler}) cannot give rise to quantities which are constant on Schwarzschild.

In the special case $p=2$, Miao \cite{miao2022mass} observed that a conformal change can be used to promote Euclidean rigidity to Schwarzschild rigidity. Indeed, if $(M,g)$ is asymptotically flat with non-negative scalar curvature and $v$ is a harmonic function asymptotic to 1 at infinity, then $(M,v^4 g)$ is also asymptotically flat with non-negative scalar curvature. Moreover, Schwarzschild can be obtained from Euclidean space via such a harmonic conformal change. By making a clever choice of conformal change, Miao \cite{miao2022mass} was thus able to derive monotone quantities for harmonic functions that are constant on Schwarzschild from the already known quantities that are constant on Euclidean space. Oronzio \cite{oronzio2022adm} dervied some further such quantities, still in the case $p=2$.

The case $p\in (1,2)\cup(2,3)$ cannot be handled by Miao's conformal change argument since $p$-harmonic conformal changes do not have to preserve non-negative scalar curvature, and Schwarzschild is not related to Euclidean space by a $p$-harmonic conformal change. Thus we instead aim to exploit the full power of inequality ({\ref 1}). To this end, we search for functions $f$, $g$, and $h$ such that 
\[
Q(t) = 4\pi(3-p)^2 f(t) + g(t)W(t) + (p-1)(3-p)h(t)\frac{dW}{dt}(t)
\]
is monotone, and constant when $M$ is Schwarzschild. In light of (\ref{1}), this amounts to solving an ODE system for $f$, $g$, and $h$. We set up and solve this system in Section \ref{section:ODE}. We single out two special choices $f_*$, $g_*$, $h_*$ and $f^*$, $g^*$, $h^*$ so that $Q_*$ decays to 0 as $t\to \infty$ and so that $Q^*$ detects that ADM mass of $M$ as $t\to \infty$. Although we do not give closed formulas for these functions, we are able to understand their initial values and their asymptotic behavior. 

In Section \ref{section:Monotonicity}, we show that the monotonicity of $Q$ persists even when $u$ is allowed to have critical points, as long as $H_2(M,\Sigma) = 0$. Here we closely follow the ideas in \cite{agostiniani2022riemannian} and \cite{agostiniani2021green}.
Finally, in Section \ref{section:asymptotics} we exploit the monotonicity of $Q_*$ and $Q^*$ to prove the sharp $p$-Penrose inequality.

%We first compute a dominating ODE inequality with Schwarzschild rigidity. Then, using this dominating ODE inequality, we construct monotone quantities whose rigidity case is the spatial Schwarzschild metric by deriving and solving  a system of ODEs to determine the coefficient functions. We also ascertain the initial values and the asymptotic behavior of the coefficient functions. Then, we give a rigorous proof of the monotonicity in presence of critical points. Finally, we obtain the asymptotic behavior of the monotone quantities, and give the proof of Corollary \ref{main-theorem}. 

\subsection{Further Discussion}

In 2007, Bray and Lee \cite{bray2009riemannian} proved that the Riemannian Penrose Inequality holds for dimension less than $8$. Since Theorem \ref{main-theorem} can be viewed as a generalization of the Riemannian Penrose Inequality, it is natural to make the following conjecture.
\begin{conjecture}
    Suppose $(M^n,g)$ is a complete, asymptotically flat manifold with non-negative scalar curvature and non-empty minimal boundary $\Sigma$, $3\leq n \le 7$. Assume that $H_{n-1}(M,\Sigma)=0$. Then
    \begin{align*}
        m_{\text{ADM}}\geq 2\left(\frac{C_p}{K_p}\right)^{\frac{n-2}{n-p}}.
    \end{align*}
    Moreover, equality holds if and only if $(M,g)$ is isometric to Schwarzschild with $\Sigma$ as the horizon.
\end{conjecture}

\begin{remark}
    Returning to the 3-dimensional setting, it would also be interesting to understand how the quantity 
    \[
    f(p) = 2\left(\frac{C_p}{K_p}\right)^{\frac{1}{3-p}}
    \]
    depends on $p$ for a fixed asymptotically flat $(M^3,g)$. If $f(p)$ is monotone, it may be possible to deduce the sharp $p$-Penrose inequality from the ordinary Riemannian Penrose Inequality. Even if this is the case, we still believe it is of interest to have a proof of the sharp $p$-Penrose inequality based purely on non-linear potential theory. 
\end{remark}

\begin{remark}
In \cite{bray2008capacity}, Bray and Miao proved an inequality between the mass of $M$ and the harmonic capacity of $\Sigma = \bd M$ in the case where $\Sigma$ is not necessarily assumed to be minimal. Likewise, Xiao \cite{xiao2016p} also proved a mass to $p$-capacity comparison in the case where $\Sigma$ is not required to be minimal. We believe our monotonicity formulas should also apply to prove this more general inequality. 
\end{remark}

% The functions $f^*$, $g^*$, and $h^*$ have the following asymptotic expansions as $t\to \infty$:
% \begin{gather*}
% f^*(t) = K_p^* e^{\frac{t}{3-p}} + (3-p) +  O(e^{-\frac{t}{3-p}}),\\
% g^*(t) = -K_p^* e^{\frac{t}{3-p}}-(3-p) - \frac{4}{3-p} + O(e^{-\frac{t}{3-p}}),\\
% h^*(t) = \frac{K_p^*}{3-p} e^{\frac{t}{3-p}} + O(e^{-\frac{t}{3-p}}).
% \end{gather*}
% Here $K_p^* = (\frac{K_p}{4\pi})^{\frac{1}{p-3}}$.  
% % Moreover, the values $f^*(0)$, $g^*(0)$, and $h^*(0)$ ensure that $Q^*(0) \ge Q^*_s(0)$ where $Q^*_s$ is the $Q^*$ function on mass 2 Schwarzschild. 
% Theorem \ref{main-theorem} is  proven by comparing $Q_*(0)$ with $\lim_{t\to\infty}Q_*(t)$ and by comparing $Q^*(0)$ with $\lim_{t\to\infty} Q^*(t)$. 

%\subsection{Organization} The remainder of the paper is organized as follows. In Section \ref{section:W-inequality}, we state and prove a differential inequality for the quantity $W$. In Section \ref{section:ODE}, we derive and solve a system of ODEs to determine the functions $f_*$, $g_*$, $h_*$ and $f^*$, $g^*$, $h^*$. In Section \ref{section:Monotonicity}, we prove the monotonicity of $Q_*$ and $Q^*$, even when the $p$-harmonic function $u$ may have critical points. Finally, in Section \ref{section:asymptotics}, we study the behavior of $Q_*$ and $Q^*$ as $t\to \infty$ and obtain the proof of Corollary \ref{main-theorem}. 

\section*{Acknowledgements}
The authors would like to thank Professor Xin Zhou for his encouragement and helpful discussions. 

X.Y. is supported by NSF grant
DMS-1945178.  
\section{A Differential Inequality for W}
\label{section:W-inequality}

In this section, we record a differential inequality satisfied by $W$ in the case where $u$ has no critical points and $\Sigma$ is a 2-sphere. Some version of this inequality is already implicitly present in \cite{agostiniani2022riemannian}, \cite{agostiniani2021green}, and \cite{oronzio2022adm}. The reader should also note that formally setting $p=1$ yields the differential inequality governing Geroch monotonicity. 

\begin{proposition}
\label{proposition:W-monotonicity}
Fix a number $1 < p < 3$. Assume that $(M^3,g)$ is a complete, asymptotically flat manifold with non-negative scalar curvature and non-empty minimal boundary $\Sigma$. Assume that $\Sigma$ is a 2-sphere and that the solution $u$ to (\ref{u-pde}) has no critical points. Then $W$ satisfies the differential inequality 
\begin{equation}
\label{W-ODE}
(p-1)(3-p)\frac{d^2W}{dt^2} \ge W - 4\pi(3-p)^2 + 2(2-p)\frac{dW}{dt} + \frac{(p-1)(5-p)}{4} \frac{\left(\frac{dW}{dt}\right)^2}{W}. 
\end{equation}
Moreover, equality holds for all $t \ge 0$ if and only if $(M,g)$ is isometric to spatial Schwarzschild with horizon boundary. 
\end{proposition}

\begin{remark}
In fact, the differential inequality (\ref{W-ODE}) is purely local, and holds for a positive $p$-harmonic function $u$ on any Riemannian manifold with non-negative scalar curvature in a region where the level sets of $u$ are connected.
\end{remark}

\begin{proof}
As shown by Moser \cite{moser2007inverse}, the function $w$ satisfies $\lap_p w = \vert \grad w\vert^p$. Using this, one computes that 
\begin{gather}
\label{equation:lap-w}
    \lap w = \vert \grad w\vert^2 - (p-2) \frac{\la \grad \vert \grad w\vert, \grad w\ra}{\vert \grad w\vert},\\
\label{equation:H}
    H = \vert \grad w\vert - (p-1) \frac{\la \grad \vert \grad w\vert, \grad w\ra}{\vert \grad w\vert^2}.
\end{gather}
The following calculations are  similar to those in \cite{agostiniani2021green}. Let $\Sigma_t = \{w=t\}$. Note that 
\[
\frac{dW}{dt} = \int_{\Sigma_t} 2\frac{\la \grad \vert \grad w\vert,\grad w\ra}{\vert \grad w\vert} + H\vert \grad w\vert \, da = W + (3-p) \int_{\Sigma_t} \frac{\la \grad \vert \grad w\vert, \grad w\ra}{\vert \grad w\vert}\, da. 
\]
This equation implies that 
\[
\int_{\Sigma_t} \frac{\la \grad \vert \grad w\vert, \grad w\ra}{\vert \grad w\vert}\, da = \frac{1}{3-p}\left(\frac{dW}{dt}-W\right)
\]
and hence that 
\begin{equation}
\label{equation:dwdt}
    \frac{dW}{dt} = \frac{2}{p-1}W - \frac{3-p}{p-1}\int_{\Sigma_t} H\vert \grad w\vert \, da. 
\end{equation}
Next observe that 
% \begin{equation}
%     \frac{d}{dt}\int_{\Sigma_t} H\vert \grad w\vert \,da \le 4\pi + \frac{1}{p-1} \int_{\Sigma_t} H\vert \grad w\vert \, da - \frac{5-p}{4(p-1)}\int_{\Sigma_t} H^2\, da. 
% \end{equation}
\begin{align*}
\frac{d}{dt} \int_{\Sigma_t} H\vert \grad w\vert \,da &= \int_{\Sigma_t} H \frac{\la \grad \vert \grad w\vert,\grad w\ra}{\vert \grad w\vert^2} + H^2 - \vert \grad w\vert\left( \lap_{\Sigma_t} \frac{1}{\vert \grad w\vert} + \left(\vert A\vert^2 + \ric(\nu,\nu)\right)\frac{1}{\vert \grad w\vert}\right)\, da\\
&= \int_{\Sigma_t} H \frac{\la \grad \vert \grad w\vert,\grad w\ra}{\vert \grad w\vert^2} + H^2 -   \frac{\vert \grad^{\Sigma_t} \vert \grad w\vert\vert^2}{\vert \grad w\vert^2} - \vert A\vert^2 - \ric(\nu,\nu)\, da. 
\end{align*}
By the Gauss equations, we have 
\begin{align*}
\int_{\Sigma_t} \vert A\vert^2 + \ric(\nu,\nu) \, da &= \int_{\Sigma_t} \frac{R}{2} - \frac{R_{\Sigma_t}}{2} + \frac{1}{2}H^2 + \frac{1}{2}\vert A\vert^2\, da = \int_{\Sigma_t} \frac{R}{2} - \frac{R_{\Sigma_t}}{2} +\frac{3}{4}H^2 + \frac{1}{2} \vert \mathring{A}\vert^2\, da. 
\end{align*}
Note that the assumption that $u$ has no critical points guarentees that each level set $\Sigma_t$ is topologically a 2-sphere. Thus we can use Gauss-Bonnet together with the previous equations to deduce that
\begin{align}
\label{equation:d2wdt2}
\frac{d}{dt} \int_{\Sigma_t} H\vert \grad w\vert \, da &\le 4\pi + \int_{\Sigma_t} \frac{1}{4}H^2+H \frac{\la \grad \vert \grad w\vert,\grad w\ra}{\vert \grad w\vert^2}\, da\\
\nonumber
&= 4\pi + \frac{1}{p-1}\int_{\Sigma_t} H \vert \grad w\vert \,da - \frac{(5-p)}{4(p-1)}\int_{\Sigma_t} H^2\, da. 
\end{align}
%&\le 4\pi + \int_{\Sigma_t} \frac{1}{4}H^2 + H \frac{\vert \grad \vert \grad w\vert\vert}{\vert \grad w\vert} + \frac{\vert \grad \vert \grad w\vert \vert^2}{\vert \grad w\vert^2} - \frac{\nu(\vert \grad w\vert)^2}{\vert \grad w\vert^2}\, da.
% \end{align*}
By (\ref{equation:dwdt}) and (\ref{equation:d2wdt2}) and H\"older's inequality, one thus obtains that 
\begin{align*}
\frac{d^2W}{dt^2} &\ge -\frac{3-p}{p-1}\left(4\pi + \frac{1}{p-1}\int_{\Sigma_t} H\vert \grad w\vert \,da - \frac{5-p}{4(p-1)}\int_{\Sigma_t} H^2 \,da\right) + \frac{2}{p-1} \frac{dW}{dt}\\
&= -\frac{4\pi(3-p)}{(p-1)}-\frac{3(3-p)}{(p-1)^2}\int_{\Sigma_t} H\vert \grad w\vert\, da + \frac{4}{(p-1)^2}\int_{\Sigma_t} \vert \grad w\vert^2\, da + \frac{(3-p)(5-p)}{4(p-1)^2}\int_{\Sigma_t} H^2\,da \\
&\ge -\frac{4\pi(3-p)}{p-1} -\frac{3(3-p)}{(p-1)^2}\int_{\Sigma_t} H\vert \grad w\vert\, da + \frac{4}{(p-1)^2}\int_{\Sigma_t} \vert \grad w\vert^2\, da + \frac{(3-p)(5-p)}{4(p-1)^2}\frac{(\int_{\Sigma_t} H\vert \grad w\vert\, da)^2}{\int_{\Sigma_t} \vert \grad w\vert^2\, da} \\
&= -\frac{4\pi(3-p)}{p-1}+\frac{1}{(p-1)(3-p)}W + \frac{2(2-p)}{(p-1)(3-p)}\frac{dW}{dt} + \frac{5-p}{4(3-p)}\frac{(\frac{dW}{dt})^2}{W}.
\end{align*}
Multiplying by $(p-1)(3-p)$ gives the inequality in the statement of the proposition. 

It remains to prove rigidity. In Schwarzschild, the level sets $\Sigma_t$ are coordinate spheres and $\vert \grad w\vert$ and $H$ are both constant on each $\Sigma_t$. Moreover, we have $R = 0$ and $\vert \mathring{A}\vert = 0$. Thus all the inequalities in the above calculation become equality. Conversely, suppose that equality holds in (\ref{W-ODE}) for all $t \ge 0$.  The vanishing of $\vert \grad^{\Sigma_t} \vert \grad w\vert \vert$ implies that $\vert \grad w\vert$ is constant on each $\Sigma_t$ and so $\vert \grad w\vert = f(w)$. Equality in H\"older's inequality implies that $H$ is proportional to $\vert \grad w\vert$ on $\Sigma_t$ and therefore $H = h(w)$ is also constant on each $\Sigma_t$.   Moreover, we have $R = 0$ and $\vert \mathring{A} \vert = 0$. Since 
\[
h'(w) = -\frac{\vert A\vert^2 + \ric(\nu,\nu)}{f(w)},
\]
it follows that $\ric(\nu,\nu)$ is constant on each $\Sigma_t$. By the Gauss equations, this implies that $R_{\Sigma_t}$ is constant on each $\Sigma_t$, and hence that each $\Sigma_t$ is isometric to a round sphere.  Thus $(M,g)$ is isometric to a warped product 
\[
\left([0,\infty)\times S^2, ds^2 + \varphi(s)^2 \, dg_0\right)
\]
where $g_0$ is the round metric on the unit sphere. Moreover, the slice $\{0\}\times S^2$ is minimal. 
% {\red I think Schwarzschild is the only warped product like this with zero scalar curvature. This should follow from uniqueness for some ODE. See the section on warped products in Petersen's book. Bray and Neves also claim something similar in section 4 of their paper. Maybe this is the easiest way to handle this?}
Since $R = 0$, the warping function $\varphi$ solves 
\[
\varphi'' = \frac{1-(\varphi')^2}{2 \varphi}. 
\]
See Petersen \cite{petersen2006riemannian} Chapter 3 Section 2 for a detailed derivation of this equation. Note that $\varphi'(0) = 0$ since $\Sigma$ is minimal and that $\varphi(0) = \sqrt{\area(\Sigma)/4\pi}$. Now let $(M_s,g_s)$ be the Schwarzschild manifold whose horizon $\Sigma_s$ satisfies $\area(\Sigma_s) = \area(\Sigma)$. Note that $(M_s,g_s)$ is  isometric to a warped product 
\[
\left([0,\infty)\times S^2,ds^2 + \psi(s)^2 \, dg_0\right),
\]
and the warping function $\psi$ also solves 
\[
\psi'' = \frac{1-(\psi')^2}{2\psi}
\]
with initial conditions $\psi'(0)=0$ and $\psi(0)=\sqrt{\area(\Sigma)/4\pi}$. It follows that $\varphi = \psi$ by uniqueness of positive solutions to the ODE. Therefore $(M,g)$ is isometric to $(M_s,g_s)$, as needed. 
\end{proof}
% Equation (\ref{equation:H}) gives that 
% \[
% h(w) = f(w) - (p-1)f'(w).
% \]
% and using the fact $h(w)$ is propotional to $f(w)$, we obtain
% \begin{align*}
%     f(w)=ce^{\frac{1-k}{p-1}w},
% \end{align*}
% where $h(w)=kf(w)$,$c$ is some positive constant.

% Since $f(w)-(3-p)f'(w)=0$ ,  we can write the metric as
% \begin{align*}
%     g=f(w)^2dw^2+h(w,\theta)_{ij}d\theta_id\theta_j.
% \end{align*}

% The equality also requires $A_{ij}=\frac{H}{2}g_{ij}$, we thus obtain
% \begin{align*}
%     \frac{\partial h_{ij}}{\partial w}=\frac{2}{|\nabla w|}A_{ij}=\frac{H}{|\nabla w|}h_{ij}=kh_{ij},
% \end{align*}
% which gives
% \begin{align*}
%     h(w,\theta)=c^{-2}e^{kw}h_0(\theta)
% \end{align*}
%  and we can then write $g=c^{-2}g_0$, where
%  \begin{align*}
%      g_0=e^{\frac{2(1-k)w}{p-1}}dw^2+e^{kw}h(\theta)_{ij}d\theta_id\theta_j
%  \end{align*}
% Since $R_g=0$, we obtain
% \begin{align*}
%     s
% \end{align*}
% which implies
% \begin{align*}
%     R_h=
% \end{align*}

\section{ODE Analysis on Schwarzschild}
\label{section:ODE}

Consider a complete, asymptotically flat manifold $(M^3,g)$ with non-negative scalar curvature. Assume that $M$ has non-empty, minimal boundary $\Sigma$. Assume additionally that 
\begin{equation} 
\label{no-critical-assumption}
\text{$\Sigma$ is a 2-sphere and the $p$-harmonic function $u$ satisfying (\ref{u-pde}) has no critical points.}
\end{equation}
We would like to find functions $f(t),g(t),h(t)$ so that the quantity 
\[
Q(t) = 4\pi(3-p)^2 f(t) + g(t)W(t) + (p-1)(3-p) h(t) \frac{dW}{dt}
\]
is non-decreasing, and constant when $M$ is Schwarzschild. In a later section, we will remove the simplifying assumption (\ref{no-critical-assumption}). 

\begin{proposition}
\label{proposition:monotone-ODE}
Suppose the functions $f,g,h\f [0,\infty)\to \R$ satisfy 
\begin{equation}
\label{fgh-ODE-t}
\begin{cases}
\left(\frac{dg}{dt} + h\right)W_s^2 + \left(g - 2(p-2)h +(p-1)(3-p)\frac{dh}{dt}\right)W_s\frac{dW_s}{dt} +  \frac{(p-1)(5-p)}{4}h\left(\frac{dW_s}{dt}\right)^2=0\\
\left(\frac{dg}{dt}+h\right)W_s^2 - \frac{(p-1)(5-p)}{4}h\left(\frac{dW_s}{dt}\right)^2 = 0,\\
\frac{df}{dt} = h,
\end{cases}
\end{equation}
where $W_s$ is the $W$ function in the model case of mass 2 Schwarzschild.  Then 
\[
Q_s(t) = 4\pi(3-p)^2 f(t) + g(t)W_s(t) + (p-1)(3-p)h(t)\frac{dW_s}{dt}(t)
\]
is constant. If, in addition, $h$ and $\frac{dg}{dt}+h$ are non-negative, then 
\[
Q(t) = 4\pi(3-p)^2 f(t) + g(t)W(t) + (p-1)(3-p)h(t)\frac{dW}{dt}(t)
\]
is non-decreasing on any asymptotically flat manifold $(M^3,g)$ with non-negative scalar curvature and $p$-harmonic function $u$ satisfying (\ref{no-critical-assumption}). 
\end{proposition} 

\begin{proof}
Using Proposition \ref{proposition:W-monotonicity} and the third equation satisfied by $f$ and $h$, one computes that 
\begin{align*}
Q'(t) &= 4\pi(3-p)^2 \frac{df}{dt} + \frac{dg}{dt}W  + g\frac{dW}{dt} + (p-1)(3-p)\frac{dh}{dt}\frac{dW}{dt} + (p-1)(3-p)h \frac{d^2W}{dt^2}\\
&\ge 4\pi(3-p)^2\left(\frac{df}{dt} -h\right) + \left(\frac{dg}{dt} + h\right)W + \frac{(p-1)(5-p)}{4}h\frac{\left(\frac{dW}{dt}\right)^2}{W}\\
& \qquad + \left(g - 2(p-2)h +(p-1)(3-p)\frac{dh}{dt}\right)\frac{dW}{dt}\\
& = \frac{1}{W}\left[\left(\frac{dg}{dt} + h\right)W^2 + \left(g - 2(p-2)h +(p-1)(3-p)\frac{dh}{dt}\right)W \frac{dW}{dt} + \frac{(p-1)(5-p)}{4}h\left(\frac{dW}{dt}\right)^2\right].
\end{align*}
Substituting the first equation satisfied by $g$ and $h$ into this inequality for $Q'$
and using that $(\ref{W-ODE})$ is equality on Schwarzschild then shows that $Q_s' \equiv  0$ on mass 2 Schwarzschild. This proves the first assertion of the proposition. 

Further assume that $h$ and $\frac{dg}{dt}+h$ are non-negative. Then re-arranging the second equation for $g$ and $h$ and taking the square root shows that 
\[
W_s\sqrt{\frac{dg}{dt}+h} = \frac{dW_s}{dt}\sqrt{\frac{(p-1)(5-p)}{4}h} . 
\]
It follows that 
\begin{align*}
0 &= \left(W_s\sqrt{\frac{dg}{dt}+h} - \frac{dW_s}{dt}\sqrt{\frac{(p-1)(5-p)}{4}h}\right)^2 \\
&= \left(\frac{dg}{dt} + h\right)W_s^2 - 2W_s\frac{dW_s}{dt}\sqrt{\frac{dg}{dt}+h} \sqrt{\frac{(p-1)(5-p)}{4}h}+  \frac{(p-1)(5-p)}{4}h\left(\frac{dW_s}{dt}\right)^2.
\end{align*}
Combined with the first equation for $g$ and $h$, this implies that 
\begin{equation}
\label{perfect-square}
g-2(p-2)h+(p-1)(3-p)\frac{dh}{dt}=-2\sqrt{\frac{dg}{dt}+h} \sqrt{\frac{(p-1)(5-p)}{4}h}
\end{equation}
and therefore that 
\begin{align*}
&\left(\frac{dg}{dt} + h\right)W^2 + \left(g - 2(p-2)h +(p-1)(3-p)\frac{dh}{dt}\right)W\frac{dW}{dt} +  \frac{(p-1)(5-p)}{4}h\left(\frac{dW}{dt}\right)^2\\
&\qquad =\left(W\sqrt{\frac{dg}{dt}+h} - \frac{dW}{dt}\sqrt{\frac{(p-1)(5-p)}{4}h}\right)^2 \ge 0,
\end{align*}
for the $W$ function on any asymptotically flat manifold $(M,g)$ satisfying the assumptions of the proposition. 
Substituting this into the previous inequality for $Q'(t)$ proves that $Q'(t) \ge 0$ for any $(M,g)$ with $p$-harmonic function $u$ satisfying the assumptions of the proposition. 
\end{proof}

\subsection{Asymptotic Expansions on Schwarzschild} In order to solve the system of ODEs (\ref{fgh-ODE-t}), we need to perform some explicit calculations with $p$-harmonic functions on the Schwarzschild manifold. Let
\[
(M_s, g_s) = \left(\R^3\setminus B_1(0), \left(1+\frac 1 r\right)^4 \delta_{ij}\right)
\]
be the spatial Schwarzschild manifold of mass 2. Note that $\Sigma_s = \bd B_1$ is a minimal surface with respect to this metric. Fix some number $1< p < 3$. We first aim to get a detailed asymptotic expansion of the solution to 
\[
\begin{cases}
\lap_p u_s = 0, &\text{on } (M_s,g_s)\\
u_s = 1, &\text{on } \Sigma_s,\\
u_s\to 0, &\text{at infinity}.
\end{cases}
\]
Note that by symmetry the solution $u_s=u_s(r)$ must be radial. Thus the problem reduces to solving the ODE 
\begin{equation}
\label{ode-p-harmonic}
(p-1)u_s''(r) + \left(\frac 2 r + \frac{2(p-3)}{r+r^2}\right)u_s'(r) = 0
\end{equation}
with boundary conditions $u_s(1) = 1$ and $u_s(r)\to 0$ as $r\to \infty$.
There is an explicit solution to this ODE in terms of hypergeometric functions. However, this offers little advantage over an asymptotic expansion. 

Thus we proceed to derive an asymptotic expansion for $u_s$ near infinity. Note that near infinity we have an expansion 
\begin{equation}
\label{expand1} 
\frac{2}{r} + \frac{2(p-3)}{r+r^2} = \frac{2}{r} - \frac{2(3-p)}{r^2}  + O(r^{-3}). 
\end{equation}
From this, we see that the ODE (\ref{ode-p-harmonic}) has a regular singular point at infinity. The equation for the indicial roots is 
\[
(p-1)\alpha(\alpha-1) + 2\alpha = 0,
\]
and this has solutions $\alpha = 0$ and $\alpha = -\frac{3-p}{p-1}$. Therefore, according to the method of Frobenius, there are two convergent series solutions 
\begin{gather*}
u_1(r) = 1 + \frac{a_1}{r} + \frac{a_2}{r^2} + \hdots,\\
u_2(r) = r^{-\frac{3-p}{p-1}}\left(1 + \frac{b_1}{r} + \frac{b_2}{r^2} + \hdots \right)
\end{gather*}
to the ODE (\ref{ode-p-harmonic}) near infinity. 
Since $u_s\to 0$ at infinity, it must be a constant multiple of $u_2$. Substituting the formula for $u_2$ into the ODE and using the expansion (\ref{expand1}), we compute that 
\[
b_1 = -\frac{(3-p)^2}{p-1}.
% \quad b_2 = \frac{(3-p)^2(5-p)}{(p-1)^2(p+1)}.
\]
Therefore, the function $u_s$ satisfies
\begin{equation}
\label{model-u-asy}
u_s(r) = c_{p,s} r^{-\frac{3-p}{p-1}}\left(1-\frac{(3-p)^2}{p-1}\frac{1}{r} %+ \frac{(3-p)^2(5-p)}{(p-1)^2(p+1)}\frac{1}{r^2} 
+ O(r^{-2})\right),
\end{equation}
where $c_{p,s}$ is chosen so that $u_s(1) = 1$. 

\begin{remark}
It is known (see \cite{benatti2022asymptotic} for example) that 
\[
c_{p,s} = \left(\frac{1}{4\pi} \int_{\Sigma_s} \vert \grad^s u_s\vert^{p-1} \, da_s \right)^{\frac{1}{p-1}}. 
\]
\end{remark}

\begin{remark} Technically the method of Frobenius is only guaranteed to succeed if the difference between the indicial roots is not an integer. However, since this is the case for an open, dense set of $p$, such solutions will in fact exist for all $p$ by continuity.
\end{remark} 

Next we compute the expansion of various other quantities related to $u_s$. First note that 
\[
u_s'(r) = c_{p,s} r^{-\frac{2}{p-1}}\left(-\frac{3-p}{p-1} + \frac{2(3-p)^2}{(p-1)^2}\frac{1}{r} %- \frac{(3-p)^2(5-p)}{(p-1)^3}\frac{1}{r^2} 
+ O(r^{-2}) \right).
\]
Define $w_s = (1-p)\log u_s$. Then the Euclidean gradient of $w_s$ satisfies 
\[
\vert \grad w_s\vert = (p-1)\frac{\vert \grad u_s\vert}{u_s} =  r^{-1}\left((3-p) - (3-p)^2\frac{1}{r} + O(r^{-2})\right). 
\]
Therefore the function 
\[
W_s(r) = \int_{\bd B_r} \vert  \grad^s w_s\vert^2 \, d a_s = \int_{\bd B_r} \vert \grad w_s\vert^2 \, da 
\]
satisfies 
\begin{gather*}
W_s(r) = 4\pi (3-p)^2\left(1 - 2(3-p) \frac{1}{r} + O(r^{-2})\right),\\
\frac{dW_s}{dr}(r) = 8\pi(3-p)^3 \frac{1}{r^2} + O(r^{-3}).
\end{gather*}
Finally define a new variable $t$ related to $r$ by $t = (1-p)\log u_s(r)$. Then 
\[
\frac{dt}{dr} = (1-p)\frac{u_s'(r)}{u_s(r)} = r^{-1}\left((3-p) - (3-p)^2 \frac{1}{r} + O(r^{-2})\right),
\]
and it follows that 
\[
\frac{dr}{dt} = r\left(\frac{1}{3-p}+ \frac{1}{r} + O(r^{-2})\right). 
\]
With these expansions at hand, we now solve the system of ODEs to determine a monotone quantity. 

Since the requirement $\frac{df}{dt} = h$ is easily satisfied, we focus attention on the first pair of equations:
\begin{equation}
\label{gh-ODE-t}
\begin{cases} 
\left(\frac{dg}{dt} + h\right)W_s^2 + \left(g - 2(p-2)h +(p-1)(3-p)\frac{dh}{dt}\right)W_s\frac{dW_s}{dt} +  \frac{(p-1)(5-p)}{4}h\left(\frac{dW_s}{dt}\right)^2=0\\ 
\left(\frac{dg}{dt}+h\right)W_s^2 - \frac{(p-1)(5-p)}{4}h\left(\frac{dW_s}{dt}\right)^2 = 0.
\end{cases}
\end{equation}
In terms of the coordinate $t = (1-p)\log(u_s(r))$ on mass 2 Schwarzschild, this system becomes 
\begin{equation}
\label{gh-ODE-r}
\begin{cases}
\left(\frac{dg}{dr}\frac{dr}{dt} + h\right)W_s^2 + \left(g - 2(p-2)h +(p-1)(3-p)\frac{dh}{dr}\frac{dr}{dt}\right)W_s\frac{dW_s}{dr}\frac{dr}{dt} +  \frac{(p-1)(5-p)}{4}h\left(\frac{dW_s}{dr}\frac{dr}{dt}\right)^2=0,\\
\left(\frac{dg}{dr}\frac{dr}{dt}+h\right)W_s^2 - \frac{(p-1)(5-p)}{4}h\left(\frac{dW_s}{dr}\frac{dr}{dt}\right)^2 = 0.
\end{cases}
\end{equation}
This can be re-written as 
\[
\begin{pmatrix}
\frac{dg}{dr} \vspace{2mm} \\  \frac{dh}{dr} 
\end{pmatrix} = \begin{pmatrix} 0 & a(r)\\ b(r) & c(r)\end{pmatrix} \begin{pmatrix} g \\ h\end{pmatrix} 
\]
where 
\begin{gather*}
a(r) = \frac{1}{\frac{dr}{dt}}\left[\frac{(p-1)(5-p)}{4W_s^2}\left(\frac{dW_s}{dr}\frac{dr}{dt}\right)^2 - 1\right], \quad 
b(r) = \frac{-1}{(p-1)(3-p)\frac{dr}{dt}},\\
c(r) = \frac{2(p-2)}{(p-1)(3-p)\frac{dr}{dt}}
 -\frac{(5-p)}{2(3-p)W_s}\frac{dW_s}{dr}.
 \end{gather*}
 Since $a,b,c$ are smooth functions, the theory of linear ODEs implies that there is a 2-dimensional family of solutions defined for all $r\in [1,\infty)$. 
 
To compute the asymptotic expansion of the solutions, it is convenient to transform the above system into a single second order equation for $g$. Indeed, the above system is equivalent to 
\begin{equation}
\label{g-ODE}
\frac{d^2g}{dr^2} = \left(\frac{1}{a}\frac{da}{dr} + c\right)\frac{dg}{dr} + a b g.
\end{equation}
Using the previously computed asymptotic expansions for $W_s$, $\frac{dW_s}{dr}$, and $\frac{dr}{dt}$, one computes that 
\begin{gather*}
a(r) = -(3-p)r\inv + (3-p)^2 r^{-2}  + O(r^{-3}),\\
b(r) = -\frac{1}{p-1}r\inv + \frac{3-p}{p-1}r^{-2} + O(r^{-3}),\\
c(r) = \frac{2(p-2)}{p-1}r\inv - \left[\frac{2(p-2)(3-p)}{p-1}+{(5-p)}\right]r^{-2} + O(r^{-3})\\
\frac{1}{a}\frac{da}{dr} = -r^{-1} +(3-p)r^{-2} + O(r^{-3}).
\end{gather*}
Therefore, equation (\ref{g-ODE}) has a regular singular point at infinity. The equation for the indicial roots is 
\[
\alpha(\alpha-1) = -\frac{3-p}{p-1}\alpha + \frac{3-p}{p-1},
\]
and this has solutions $\alpha = 1$ and $\alpha = - \frac{3-p}{p-1}$.  Therefore, according to the method of Frobenius, there are two convergent series solutions 
\begin{gather*}
g_1(r) = r\left(-1+\frac{a_1}{r}+\frac{a_2}{r^2} + \hdots \right),\\
g_2(r) = r^{-\frac{3-p}{p-1}}\left(-1+\frac{b_1}{r}+\frac{b_2}{r^2} + \hdots\right). 
\end{gather*}
We will now investigate some further properties of these solutions.

\subsection{The Decaying Solution} First we investigate the decaying solution $g_2$. Let 
\[
h_2 = \frac{1}{a} \frac{dg_2}{dr} = r^{-\frac{3-p}{p-1}}\left(\frac{1}{p-1} + O(r\inv)\right)
\] 
so that $g_2$ and $h_2$ solve the system of ODEs. 

\begin{proposition}
We have that $h_2 > 0$ for all  $r\in[1,\infty)$, and $\frac{dg_2}{dr}\frac{dr}{dt} + h_2 > 0$ for all $r\in [1,\infty)$. 
\end{proposition} 

\begin{proof}
It is clear from the asymptotics that $h_2 > 0$ in a neighborhood of infinity. Suppose for contradiction that $h_2(r)\le 0$ somewhere on $[1,\infty)$.   Let 
$
\rho = \sup\{r\in[1,\infty): h_2(r) \le 0\}.
$
Note that 
\[
h_2(\rho) = 0 \quad\text{and} \quad \frac{dh_2}{dr}(\rho) \ge 0.
\]
 Substituting $h_2(\rho)=0$ into the ODE system (\ref{gh-ODE-r}), we see that
\begin{gather*}
\quad \frac{dg_2}{dr}(\rho) = 0, \\ \quad g_2(\rho) + (p-1)(3-p)\frac{dh_2}{dr}(\rho) \frac{dr}{dt}(\rho) = 0.
\end{gather*}
In particular, substituting $\frac{dh_2}{dr}(\rho)\ge 0$ into the second equation above, we deduce that  $g_2(\rho) \le 0$.  Now define 
\[
f_2(r) = -\int_r^\infty \frac{h_2(s)}{\frac{dr}{dt}(s)} \, ds   
\]
which is convergent by the asymptotic expansion for $h_2$. Observe that $f_2(r) < 0$ for $r\ge \rho$. We have 
\[
\frac{df_2}{dr}\frac{dr}{dt} =  h_2
\]
and therefore the quantity 
\[
Q_s(r) = 4\pi(3-p)^2 f_2(r) + g_2(r) W_s(r) + (p-1)(3-p)h_2(r)\frac{dW_s}{dr}\frac{dr}{dt}
\]
is constant. Since $Q_s(r)\to 0$ as $r\to \infty$, it follows that $Q_s(r) \equiv 0$. Thus 
\[
0 = Q_s(\rho) = 4\pi(3-p)^2 f_2(\rho) + g_2(\rho)W_s(\rho),
\]
and this is a contradiction since $f_2(\rho) < 0$ and $g_2(\rho) \le 0$. This proves that $h_2 > 0$ on $[1,\infty)$. 
Next observe that 
\[
\frac{dg_2}{dr}\frac{dr}{dt} + h_2 = \left(a\frac{dr}{dt}+1\right)h_2
\]
and so $\frac{dg_2}{dr}\frac{dr}{dt} + h_2$ has the same sign as $h_2$. Thus it is also strictly positive. 
\end{proof}

\begin{remark}
Note that $g_2$ and $h_2$ are really functions of $t$ where $t$ is related to $r$ by $t = (1-p)\log(u_s(r))$. They also depend on $p$ but we have chosen to suppress this in the notation. 
\end{remark}

\begin{definition} Define $g_*(t) = g_2(t)$ and $h_*(t) = h_2(t)$ and $f_*(t) = -\int_t^\infty h_*(s)\, ds$. Note that $f_*(t) \to 0$, $g_*(t)\to 0$, and $h_*(t)\to 0$ as $t\to \infty$.
\end{definition}

\begin{definition} 
For a given asymptotically flat manifold $(M,g)$, let 
\[
Q_*(t) = 4\pi(3-p)^2 f_*(t) + g_*(t)W(t) + (p-1)(3-p)h_*(t) \frac{dW}{dt}(t). 
\]
\end{definition}

% \begin{proposition} Let $(M^3,g)$ be an asymptotically flat manifold with non-negative scalar curvature and minimal boundary $\Sigma$. Fix a number $1 < p < 2$. Assume that $\Sigma$ is a 2-sphere and that the solution $u$ to (\ref{pde}) has no critical points. Define the quantity 
% \[
% X_*(t) = 4\pi(3-p)^2 f_*(t) + g_*(t) W(t) + (p-1)(3-p)h_*(t)\frac{dW}{dt}(t). 
% \]
% Then $X_*$ is an increasing function and $X_*(t) \to 0$ as $t\to \infty$. Moreover, if $X_*$ is constant then $(M,g)$ is isometric to Schwarzschild with horizon boundary. 
% \end{proposition}
% \begin{proof}
% This follows from the preceding discussion. 
% \end{proof}

\begin{proposition}
\label{proposition:decaying-value-0}
At time $t=0$, one has $f_*(0) < 0$, and $g_*(0) + 2(3-p)h_*(0) > 0$, and 
\[
\frac{-4\pi(3-p)^2 f_*(0)}{g_*(0)+2(3-p)h_*(0)} = W_s(0).
\] 
\end{proposition}

\begin{proof}
Consider the model case of mass 2 Schwarzschild. Then the function $Q_*$ is constant and $Q_*(t) \to 0$ as $t\to \infty$ and so $Q_*(0) = 0$.  It is clear from the definition of $f_*$ and the fact that $h_* > 0$ that 
\[
f_*(0) < 0. 
\]
Since $\Sigma$ is a minimal surface, equation (\ref{equation:dwdt}) implies that $\frac{dW_s}{dt}(0) = \frac{2}{p-1}W_s(0)$. Therefore 
\[
0 = Q_*(0) = 4\pi(3-p)^2 f_*(0) + \left(g_*(0) + 2(3-p)h_*(0)\right)W_s(0).
\]
It follows that $g_*(0)+2(3-p)h_*(0)>0$ and that 
\[
\frac{-4\pi(3-p)^2f_*(0)}{g_*(0)+2(3-p)h_*(0)} = W_s(0),
\]
as needed. 
\end{proof}

We can now state the following preliminary version of Theorem \ref{theorem:decaying}. 

\begin{proposition}
Let $(M^3,g)$ be a complete, asymptotically flat manifold with non-negative scalar curvature and non-empty, minimal boundary $\Sigma$. Assume in addition that (\ref{no-critical-assumption}) holds. Then $Q_*$ is non-decreasing, and $Q_*$ is constant if and only if $(M^3,g)$ is isometric to Schwarzschild with horizon boundary. Finally, one has $f_*(0) < 0$ and $g_*(0) + 2(3-p)h_*(0) > 0$. 
\end{proposition}

\begin{proof}
The monotonicity is a consequence of the fact that $f_*$, $g_*$, and $h_*$ satisfy the assumptions of Proposition \ref{proposition:monotone-ODE}.  If $Q_*$ is constant, then it follows from the proof of Proposition \ref{proposition:monotone-ODE} that equality holds in (\ref{W-ODE}) for all $t\ge 0$. The rigidity then follows from the rigidity in Proposition \ref{proposition:W-monotonicity}. The information about the initial conditions is contained in Proposition \ref{proposition:decaying-value-0}.
\end{proof}

% \begin{corollary}
% Let $(M^3,g)$ be an asymptotically flat manifold with non-negative scalar curvature and minimal boundary $\Sigma$. Fix a number $1 < p < 2$. Assume that $\Sigma$ is a 2-sphere and that the solution $u$ to (\ref{pde}) has no critical points.  Then $W(0) \le W_s(0)$. If equality holds, then $(M,g)$ is isometric to Schwarzschild with horizon boundary. 
% \end{corollary}

% \begin{proof}
% Since $X_*$ is increasing, we have $X_*(0) \le \lim_{t\to \infty} X_*(t) = 0$.  Since $\Sigma$ is minimal, we have $\frac{dW}{dt}(0) = \frac{2}{p-1}W(0)$. Now observe that 
% \[
% X_*(0) = 4\pi(3-p)^2 f_*(0) + \big(g_*(0) + 2(3-p)h_*(0)\big)W(0). 
% \]
% Since $X_*(0) \le 0$ and $f_*(0)<0$ and $g_*(0) + 2(3-p)h_*(0) > 0$, this implies that 
% \[
% W(0) \le \frac{-4\pi(3-p)^2f_*(0)}{g_*(0)+2(3-p)h_*(0)} = W_s(0).
% \]
% This completes the proof. 
% \end{proof}

\subsection{The Growing Solution} Next we turn attention to the growing solution $g_1$. For applications, we need to compute the next coefficient $a_1$ in the asymptotic expansion for $g_1$.  Substituting the series for $g_1$ into the equation (\ref{g-ODE}) and using the expansions for $a,b,c$ we find that 
\[
a_1 = \left(\frac{p-1}{3-p}\right)\left(\left[(3-p) - \frac{2(p-2)(3-p)}{p-1} - (5-p)\right] -\frac{2(3-p)^2}{p-1}\right)= -\frac{4}{p-1}. 
\]
Thus we have 
\[
g_1(r) = -r - \frac{4}{3-p} + O(r\inv). 
\]
Define 
\[
h_1(r) = \frac{1}{a}\frac{dg_1}{dr} = \frac{r}{3-p} + 1 + O(r\inv)
\]
so that $g_1$ and $h_1$ solve the system of ODEs (\ref{gh-ODE-r}). 

 In the following, we fix once and for all a small number $\eps >0$. This choice of $\eps$ is not canonical, and in principle one obtains monotone quantities for each different choice of $\eps$.  However, the precise value of $\eps$ turns out to be unimportant for our main application to the sharp $p$-Penrose inequality, and so we have not investigated this dependence further.

\begin{proposition}
%Fix a small number $\eps >0$. 
Let $g(r),h(r)$ be the unique solution to the system of ODEs (\ref{gh-ODE-r}) with initial condition $g(1)=-1$ and $h(1)=\eps$. 
Then $h(r) > 0$ and $g(r) < 0$ for all $r \in [1,\infty)$. 
\end{proposition}

\begin{proof}
First we show that if $g(r) < 0$ on $[1,R]$ then $h(r) > 0$ on $[1,R]$. Suppose this is not the case. Then we can let $\rho = \inf\{r\in[1,R]:\, h(r) \le 0\}$. Note that $\rho > 1$. We must have 
\[
h(\rho) = 0\quad \text{and}\quad \frac{dh}{dr}(\rho)\le 0. 
\] 
According to the ODE this implies that 
\[
\frac{dg}{dr}(\rho) = 0 \quad \text{and}\quad g(\rho) + (p-1)(3-p)\frac{dh}{dr}(\rho) \frac{dr}{dt}(\rho) = 0. 
\]
It follows that $g(\rho) \ge 0$, a contradiction. This proves the claim. 

Hence to prove the proposition, it suffices to show that $g(r) < 0$ for all $r\in [1,\infty)$. Suppose for contradiction that this is not the case. Then we can let 
\[
\rho=\inf\{r\in [1,\infty):\, g(r) \ge 0\}.
\]
Note that $\rho > 1$. According to the previous claim, we have $h(r) \ge 0$ for all $r\in[1,\rho]$.   
Define 
\[
f(r) =  \int_1^r \frac{h(s)}{\frac{dr}{dt}(s)}\, ds
\]
and note that $f(1) = 0$ and $f(\rho) > 0$. Since $g$ and $h$ solve the ODE and $f$ is an antiderivative for $h$ with respect to $t$, the quantity 
\[
Q_s(r) = 4\pi(3-p)^2 f(r) + g(r) W_s(r) + (p-1)(3-p)h(r) \frac{dW_s}{dr}(r) \frac{dr}{dt}(r)
\]
is conserved. We have 
\[
Q_s(1) = -W_s(1) + \eps (p-1)(3-p) \frac{dW_s}{dr}(1)\frac{dr}{dt}(1),
\]
and therefore $Q_s(1)$ is negative, provided $\eps$ is chosen sufficiently small. On the other hand, we have 
\begin{align*}
Q_s(\rho) &= 4\pi(3-p)^2 f(\rho) + (p-1)(3-p)h(\rho) \frac{dW_s}{dr}(\rho)\frac{dr}{dt}(\rho) > 0,
\end{align*}
and this is a contradiction. Thus $g(r) < 0$ for all $r\in [1,\infty)$, as needed. 
\end{proof} 

\begin{proposition} 
\label{proposition:growing-value-0}
There exist solutions $g(r)$ and $h(r)$ to the ODE (\ref{gh-ODE-r}) such that 
\[
g(1) + 2(3-p)h(1) < 0,
\]
and $h > 0$ and $\frac{dg}{dr}\frac{dr}{dt} + h > 0$ for all $r\in [1,\infty)$, and moreover 
\begin{gather*}
g(r) = -r -\frac{4}{3-p} + O(r\inv)\,\\
h(r) = \frac{r}{3-p}+1+O(r\inv),
\end{gather*} 
as $r\to \infty$. 
\end{proposition} 

\begin{proof}
Let $\tilde g(r)$ and $\tilde h(r)$ be the solution to the ODE (\ref{gh-ODE-r}) with initial conditions $\tilde g(1) = -1$ and $\tilde h(1) = \eps$. Then we know that 
\[
\tilde g = c_1 g_1 + c_2 g_2, \quad \tilde h = c_1 h_1 + c_2h_2
\]
for some constants $c_1$ and $c_2$. By perturbing $\eps$ slightly if necessary, we can ensure that $c_1 \neq 0$. According to the previous proposition, we have $\tilde g < 0$ and $\tilde h > 0$. Note that $\frac{d\tilde g}{dr}\frac{dr}{dt}+\tilde h$ always has the same sign as $\tilde h$. By the asymptotic expansions of $g_1,g_2,h_1,h_2$, it must be that $c_1 > 0$ and 
\begin{gather*}
\tilde g = c_1\left(-r - \frac{4}{3-p} + O(r\inv)\right),\\
\tilde h = c_1\left(\frac{r}{3-p} + 1 + O(r\inv)\right),\\
\end{gather*}
as $r\to \infty$. The result now follows with $g(r) = \tilde g(r)/c_1$ and $h(r) = \tilde h(r)/c_1$. 
\end{proof}

\begin{remark}
As before, the functions $g$ and $h$ are better thought of as functions of $t$ where $t = (1-p)\log u_s(r)$. The functions $g$ and $h$ also depend on $p$, but we have not indicated this in the notation.
\end{remark}

\begin{definition}
\label{def:f}
Define $g^*(t) = g(t)$ and $h^*(t) = h(t)$ where $g$ and $h$ are given by the previous proposition. Also define 
\[
f^*(t) = \int_0^t h^*(s)\, ds + q 
\]
where $q$ is chosen so that the asymptotic expansion of $f^*$ at infinity has constant order term $3-p$.
\end{definition}

\begin{definition}
For a given asymptotically flat manifold $(M,g)$, define 
\[
Q^*(t) = 4\pi(3-p)^2 f^*(t) + g^*(t)W(t) + (p-1)(3-p)h^*(t)\frac{dW}{dt}(t). 
\]
\end{definition}

\begin{remark}\label{remark: asy behavior of coefficient functions}
As $t = (1-p)\log u_s(r)$, we have 
\[
r = \tilde c_{p,s} e^{\frac{t}{3-p}} - (3-p) + O(e^{-\frac{t}{3-p}}),
\]
where $\tilde c_{p,s} = (c_{p,s})^{\frac{p-1}{3-p}}$. It follows that we have expansions 
\begin{gather*}
f^*(t) = \tilde c_{p,s} e^{\frac{t}{3-p}} + (3-p) +  O(e^{-\frac{t}{3-p}}),\\
g^*(t) = -\tilde c_{p,s} e^{\frac{t}{3-p}}-(3-p) - \frac{4}{3-p} + O(e^{-\frac{t}{3-p}}),\\
h^*(t) = \frac{\tilde c_{p,s}}{3-p} e^{\frac{t}{3-p}} + O(e^{-\frac{t}{3-p}}),
\end{gather*}
as $t\to \infty$. 
\end{remark}

We can now state the following preliminary version of Theorem \ref{theorem:growing}. 

\begin{proposition}
Let $(M^3,g)$ be a complete, asymptotically flat manifold with non-negative scalar curvature and non-empty, minimal boundary $\Sigma$. Assume in addition that (\ref{no-critical-assumption}) holds. Then $Q^*$ is non-decreasing, and $Q^*$ is constant if and only if $(M^3,g)$ is isometric to Schwarzschild with horizon boundary. Finally, one has $g^*(0) + 2(3-p)h^*(0) < 0$. 
\end{proposition}

\begin{proof}
The monotonicity is a consequence of the fact that $f^*$, $g^*$, and $h^*$ satisfy the assumptions of Proposition \ref{proposition:monotone-ODE}. The rigidity follows from the rigidity in Proposition \ref{proposition:W-monotonicity}. The information about the initial conditions is contained in Proposition \ref{proposition:growing-value-0}.
\end{proof}

\section{Monotonicity in the Presence of Critical Points}
\label{section:Monotonicity}

Throughout this section, we assume that $(M^3,g)$ is a complete, asymptotically flat manifold with non-negative scalar curvature and non-empty, minimal boundary $\Sigma$. We further assume that $H_2(M,\Sigma) = 0$ and let $u$ be the $p$-harmonic function satisfying (\ref{u-pde}). Define $w$ and $W$ as before and let $\Sigma_t = \{w=t\}$. We assume that $f,g,h\f [0,\infty)\to \R$ solve the ODE system (\ref{fgh-ODE-t}) and that $h$ and $\frac{dg}{dt}+h$ are positive. For each regular value $t$ of $w$, define 
\begin{align*}
Q(t) &= 4\pi(3-p)^2 f(t) + g(t)W(t) + (p-1)(3-p)h(t)\frac{dW}{dt}(t)\\
&= 4\pi(3-p)^2 f(t) + \big(g(t) + 2(3-p)h(t)\big)\int_{\Sigma_t} \vert \grad w\vert^2 \, d\sigma_t - (3-p)^2 h(t) \int_{\Sigma_t} H\vert \grad w\vert\, d\sigma_t.
\end{align*}
The goal of this section is to prove the monotonicity of $Q$, even when $u$ is allowed to have critical points. Following \cite{agostiniani2022riemannian}, we give the proof in a series of steps of increasing generality. We shall be relatively brief since similar computations have been done in \cite{agostiniani2022riemannian} and \cite{agostiniani2021green}.  

\subsection{No Critical points} $ $
First consider the simplest case when $w$ has no critical points on the entire manifold $M$.
Define the vector field  
\begin{align*}
    X=\ &4\pi (3-p)^2f(w)e^{-w}c_0^{-1}|\nabla w|^{p-2}\nabla w+(g(w)+2(3-p)h(w))|\nabla w|\nabla w\\
      &+(3-p)^2h(w)\left(-\Delta w\frac{\nabla w}{|\nabla w|}+\nabla|\nabla w|\right),
\end{align*}
where $c_0=\int_{\Sigma_0}|\nabla w|^{p-1}$. By the same methods as \cite{agostiniani2022riemannian} Section 1.2, we compute that
\begin{align*}
    \div X=\ &4\pi(3-p)^2f'(w)e^{-w}c_0^{-1}|\nabla w|^p\\
                &+\left[ g(w)+\frac{(3-p)(5+p)}{4}h(w)+g'(w)+(3-p)(p-1)h'(w)\right]|\nabla w|^3\\
                &+(3-p)\left[g(w)+\frac{(3-p)(p+1)}{2}h(w)+(3-p)(p-1)h'(w)\right]\langle\nabla|\nabla w|,\nabla w\rangle\\
                &+\frac{(p-1)(5-p)(3-p)^2}{4}h(w)\frac{|\nabla^{\perp}|\nabla w||^2}{|\nabla w|}\\
                &+\frac{(3-p)^2}{2}h(w)|\nabla w|(R_g-R_{\Sigma_t}+\|\mathring{A}\|^2)+(3-p)^2h(w)\frac{|\nabla^{\Sigma}|\nabla w||^2}{|\nabla w|}\\
            \geq\ &4\pi(3-p)^2 f'(w)e^{-w}c_0^{-1}|\nabla w|^p-(3-p)^2h(w)|\nabla w|K_{\Sigma_t}\\
                &+\left[a(w)\nabla w+b(w)\frac{\nabla^{\perp}|\nabla w|}{|\nabla w|}\right]^2|\nabla w|,
\end{align*}
where 
\begin{align*}
    a(w) &= \sqrt{g'(w)+h(w)}-\sqrt{\frac{(p-1)(5-p)}{4}h(w)},\\
    b(w) &= - (3-p) \sqrt{\frac{(p-1)(5-p)}{4}h(w)}.
\end{align*}
% \begin{align*}
%     a(w)^2&=g(w)+\frac{(3-p)(5+p)}{4}h(w)+g'(w)+(3-p)(p-1) h'(w),\\ b(w)^2&=\frac{(p-1)(5-p)(3-p)^2}{4}h(w).
% \end{align*}
To make the perfect square, we used the ODEs which are satisfied by $f$, $g$, and $h$ and in particular the relation (\ref{perfect-square}). 

Now observe that 
\begin{align*}
    Q(t)=\int_{\Sigma_t}\langle X,\frac{\nabla w}{|\nabla  w|}\rangle \,d\sigma_t.
\end{align*}
Using the Gauss-Bonnet formula and the fact that level sets of $w$ are connected, for any $t_2>t_1$, we have
\begin{align*}
    Q(t_2)-Q(t_1)&=\int_{\Sigma_{t_2}}\langle X,\frac{\nabla w}{|\nabla w|}\rangle \, d\sigma_{t_2}-\int_{\Sigma_{t_1}}\langle X,\frac{\nabla w}{|\nabla w|}\rangle \, d\sigma_{t_1}\\
                     &=\int_{\{t_1\leq w\leq t_2\}}\div X\, dV\\
                     &=\int_{t_1}^{t_2}\int_{\Sigma_t}\frac{\div X}{|\nabla w|}d\sigma_t \geq 0.
\end{align*}
This proves the monotonicity of $Q$. 

\subsection{Monotonicity with negligible critical values}\label{subsection:case2}
Next we aim to prove the monotonicity of $Q$ in the case where the set of critical values of $w$ has measure 0.
Consider the vector field $X$ from the previous subsection, and write 
\begin{align*}
    X=4\pi (3-p)^2f(w)e^{-w}c_0^{-1}|\nabla w|^{p-2}\nabla w+Y.
\end{align*}
Note that $Y$ is only well-defined away from the critical points of $w$. Consider a sequence of cut-off functions $\eta_k:[0,\infty)\to [0,1]$ satisfying
\[
	\eta_k(\tau)\equiv 0 \quad \forall \tau\in [0,\frac{1}{2k}], \quad 0\leq \eta_k'(\tau)\leq 2k\quad \forall \tau\in [\frac{1}{2k},\frac{3}{2k}],\quad \eta_k(\tau)\equiv 1\quad \forall \tau\in [\frac{3}{2k},\infty).
\]
 Using the cut-off functions, we introduce the vector fields
 \[
	Y_k=\eta_k(\alpha(w)|\nabla w|)Y,
 \]
 where $\alpha$ is a function satisfying the ODE
 \begin{align}\label{ODE}
    \alpha(g+(3-p)(p-1)h)=(3-p)^2(p-1)\alpha' h.
 \end{align}
More precisely, we can solve this ODE in terms of $g,h$ as
 \begin{align*}
    \alpha(w)=\exp\left({\int_0^w\frac{g+(3-p)(p-1)h}{(3-p)^2(p-1)h}dt}\right).
 \end{align*}
Then we define 
\begin{align*}
    X_k=4\pi (3-p)^2f(w)e^{-w}c_0^{-1}|\nabla w|^{p-2}\nabla w+Y_k.
\end{align*}
By direct computation, we obtain
\begin{align*}
    \div X_k\geq 4\pi(3-p)^2 f'(w)e^{-w}c_0^{-1}|\nabla w|^p+\eta_k(\alpha(w)|\nabla w|)\div Y.
\end{align*}
Now, for given regular values $0\leq s\leq\tau\leq \infty$, we can find  $k\in\mathbb{N}$ large enough that
\begin{align*}
    Q(\tau)-Q(s)&=\int_{\Sigma_{\tau}}\langle X_k,\frac{\nabla w}{|\nabla w|}\rangle \, d\sigma_{\tau}-\int_{\Sigma_s}\langle X_k,\frac{\nabla w}{|\nabla w|}\rangle \, d\sigma_{s}\\
                    &=\int_{M_{s,\tau}}\div X_k\, dV\\
                    &\geq\int_{M_{s,\tau}}4\pi(3-p)^2 f'(w)e^{-w}c_0^{-1}|\nabla w|^p+\eta_k(\alpha(w)|\nabla w|)\div Y\, dV.
\end{align*}
Note that, away from the critical points of $w$, we have
\begin{align*}
    \div Y&=\left[ g(w)+\frac{(3-p)(5+p)}{4}h(w)+g'(w)+(3-p)(p-1)h'(w)\right]|\nabla w|^3\\
    &+\frac{3-p}{p-1}\left[g(w)+\frac{(3-p)(p+1)}{2}h(w)+(3-p)(p-1)h'(w)\right](|\nabla w|^2-|\nabla w|H)\\
    &+\frac{(p-1)(5-p)(3-p)^2}{4}h(w)\frac{|\nabla^{\perp}|\nabla w||^2}{|\nabla w|}\\
    &+\frac{(3-p)^2}{2}h(w)|\nabla w|(R_g-R_{\Sigma}+\|\mathring{A}\|^2)+(3-p)^2h(w)\frac{|\nabla^{\Sigma}|\nabla w||^2}{|\nabla w|}.
\end{align*}
This can be divided into two parts: a non-negative part and a part which is bounded on any compact subset of $M$.
Applying the monotone convergence theorem and dominated convergence theorem, and passing $k\to\infty$, we obtain that
\begin{align*}
    Q^*(\tau)-Q^*(s)&\geq \int_{M_{s,\tau}}4\pi(3-p)^2 f'(w)e^{-w}c_0^{-1}|\nabla w|^p+\int_{M_{s,\tau}}\chi_{M\setminus \text{Crit}(w)}\div Y.
\end{align*}
The assumption that $H_2(M,\Sigma)=0$ implies that all the regular level sets of $u$ are connected. Combining this with the co-area formula and the Gauss-Bonnet theorem, we get that the right hand side of the previous inequality is non-negative. This is the desired monotonicity.

\subsection{An approximate monotonicity formula}
Finally we consider the general case, with no assumption on the critical values of $w$.  Since the best possible regularity for $p$-harmonic functions, when $p\neq 2$,  is $C^{1,\alpha}$, one cannot directly apply Sard's theorem and claim that the critical values of $w$ are negligible. To prove the monotonicity formula without assumption on the critical values of $w$, we use an $\epsilon$-approximation to prove an approximate monotonicity formula and obtain the desired result after passing $\epsilon\to 0$. 

In this subsection, we use $w_p$ to denote $(1-p)\log u$. Recall that $w_p$ solves the equation $\lap_p w_p = \vert \grad w_p\vert^p$. 
Consider the locally perturbed version of this equation:
\begin{equation}\label{espilonveq}
    \begin{cases}
    \text{div}((\sqrt{|\nabla w|^2+\epsilon^2})^{p-2}\nabla w)=(\sqrt{|\nabla w|^2+\epsilon^2})^{p-2}|\nabla w|^2	&   \text{ in } M_T=\{0\leq w_p\leq T\},\\
    w=0 & \text{ on } \Sigma,\\
    w=T & \text{ on } \{w_p=T\}
    \end{cases}
    \end{equation}
where $T\in (0,\infty)$ is a fixed regular value of $w_p$.

Similarly to $w_p$, the locally perturbed function $w$ satisfies
\begin{align}
    \int_{\Sigma_t^{\epsilon}}(\sqrt{|\nabla w|^2+\epsilon^2})^{p-2}|\nabla w|=c_{\epsilon,0}e^t,
\end{align}
where $c_{\epsilon,0}=\int_{\Sigma_0^{\epsilon}}(\sqrt{|\nabla w|^2+\epsilon^2})^{p-2}|\nabla w|\, d\sigma_t^{\epsilon}$, and we use $\Sigma_t^{\epsilon}$ to denote the level sets of \eqref{espilonveq}.

We then consider the vector field
\begin{align*}
    X_{\epsilon}=4\pi (3-p)^2f(w)e^{-w}c_{\epsilon,0}^{-1}(\sqrt{|\nabla w|^2+\epsilon^2})^{p-2}\nabla w+Y_{\epsilon},
\end{align*}
where
\begin{align*}
    Y_{\epsilon}=(g(w)+2(3-p)h(w))|\nabla w|\nabla w+(3-p)^2h(w)\left(-\Delta w\frac{\nabla w}{|\nabla w|}+\nabla|\nabla w|\right).
\end{align*}
As in previous subsections, we compute
\begin{align*}
    \div X_{\epsilon}=4\pi (3-p)^2 f'(w)e^{-w}c_{\epsilon,0}^{-1}(\sqrt{|\nabla w|^2+\epsilon^2})^{p-2}|\nabla w|^2+\div Y_{\epsilon},
\end{align*}
and
\begin{align*}
    \div Y_{\epsilon}=&|\nabla w|\left[a(w)\nabla w+b(w)\frac{|\nabla w|}{|\nabla w|^2+\epsilon^2}\nabla^{\perp}|\nabla w|\right]^2-\frac{\epsilon^2(3-p)^2h|\nabla w|^3}{2(p+1)|\nabla w|^2+3\epsilon^2}\\
                        +&\frac{\epsilon^2(3-p)^2h}{|\nabla w|^2+\epsilon^2}\left[\sqrt{\frac{2(p+1)|\nabla w|^2+3\epsilon^2}{4(|\nabla w|^2+\epsilon^2)|\nabla w|}}\nabla^{\perp}|\nabla w|-\sqrt{\frac{(|\nabla w|^2+\epsilon^2)|\nabla w|}{2(p+1)|\nabla w|^2+3\epsilon^2}}\nabla w\right]^2\\
                        +& \frac{(3-p)^2}{2}|\nabla w|h(R_g-R_{\Sigma}+\|\mathring{A}\|^2)+(3-p)^2h\frac{|\nabla^{\Sigma}|\nabla w||^2}{|\nabla w|}.
\end{align*}
Using the same cut-off functions as the previous subsection, we define
\begin{align*}
    Y_{k,\epsilon}&=\eta_k(\alpha(w)|\nabla w|)Y_{\epsilon},\\
    X_{k,\epsilon}&=4\pi (3-p)^2f(w)e^{-w}c_{\epsilon,0}^{-1}(\sqrt{|\nabla w|^2+\epsilon^2})^{p-2}\nabla w+Y_{k,\epsilon}.
\end{align*}
Note that, when $\eta_k'\neq 0$, we have
\begin{align*}
    \frac{\langle \nabla \eta_k, Y_{\epsilon}\rangle}{\eta_k'}&=\alpha'(g+2(3-p)h)|\nabla w|^4+\alpha'(3-p)^2h(-\Delta w|\nabla w|^2+|\nabla w|\langle \nabla|\nabla w|,\nabla w\rangle)\\
                                              &+(\alpha)(g+2(3-p)h)|\nabla w|\langle\nabla|\nabla w|,\nabla w\rangle\\
                                              &+(\alpha)(3-p)^2h(-\Delta w\langle \nabla |\nabla w|,\frac{\nabla w}{|\nabla w|}\rangle +|\nabla|\nabla w||^2)\\
                                              &\geq (2-p)(3-p)^2\epsilon^2\frac{\langle \nabla|\nabla w|,\nabla w\rangle}{|\nabla w|(|\nabla w|^2+\epsilon^2)}\\
                                              &\geq -|2-p|(3-p)^2h\max_{M_T}|\nabla^2 w|.
\end{align*}
Therefore, we obtain that
\begin{align*}
    \div Y_{k,\epsilon}\geq \eta_k(\alpha(w)|\nabla w|)\text{div}Y_{\epsilon}-\frac{5}{k}|2-p|(3-p)^2h\max_{M_T}|\nabla^2 w|.
\end{align*}
For regular values $t$ of $w$, define
\begin{align}
    Q_{k,\epsilon}(t):=\int_{\Sigma_t^{\epsilon}}\langle X_{k,\epsilon},\frac{\nabla w}{|\nabla w|}\rangle.
\end{align}
Now, for any regular values $0\leq s\leq\tau< \infty$, and large enough $k$, we have
\begin{align*}
    Q_{k,\epsilon}(\tau)-X_p^{k,\epsilon}(s)=&\int_{\Sigma_{\tau}^{\epsilon}}\langle X_{k,\epsilon},\frac{\nabla w}{|\nabla w|}\rangle \, d\sigma-\int_{\Sigma_s^{\epsilon}}\langle X_{k,\epsilon},\frac{\nabla w}{|\nabla w|}\rangle \, d\sigma\\
                                              =&\int_{\{s\leq w\leq \tau\}}\div X_{k,\epsilon}\\
                                           \geq& \int_{\{s\leq w\leq \tau\}}4\pi (3-p)^2 f'(w)e^{-w}c_{\epsilon,0}^{-1}(\sqrt{|\nabla w|^2+\epsilon^2})^{p-2}|\nabla w|^2\\
                                               &+ \int_{\{s\leq w\leq \tau\}}\eta_k(\alpha(w)|\nabla w|)\div Y_{\epsilon}-\frac{5}{k}|2-p|(3-p)^2h\max_{M_T}|\nabla^2 w|.
\end{align*}

By the same argument as in the previous subsection \ref{subsection:case2}, we pass $k\to\infty$, and obtain
\begin{align*}
    Q_{\epsilon}(\tau)-Q_{\epsilon}(s)\geq& \int_{\{s\leq w\leq \tau\}\setminus\text{Crit}(w)}\int_{\Sigma_t^{\epsilon}}\frac{\div X_{\epsilon}}{|\nabla w|}\\
                                          \geq&-\epsilon^2(3-p)^2\int_{\{s\leq w\leq\tau\}\setminus\text{Crit}(w)}\frac{h|\nabla w|^3}{2(p+1)|\nabla w|^2+3\epsilon^2}\\
                                          \geq&-\frac{\epsilon(3-p)^2}{6}\int_{\{s\leq w\leq\tau\}\setminus\text{Crit}(w)}|\nabla w|^2,
\end{align*}
where for regular values $t$, we define
\begin{align}
    Q_{\epsilon}:=\int_{\Sigma_t^{\epsilon}}\langle X_{\epsilon}, \frac{\nabla w}{|\nabla w|}\rangle.
\end{align}
Using the same argument as in \cite{agostiniani2022riemannian}, we obtain the following lemma.
\begin{lemma}
    For regular values $t$, 
    \begin{align*}
        \lim_{\epsilon\to 0}Q_{\epsilon}(t)=Q(t).
    \end{align*}
\end{lemma}

\noindent With this lemma, we obtain the monotonicity formula:

\begin{proposition}
    For regular values $0\leq s\leq\tau<\infty$, we have
    \begin{align}
        Q(\tau)\geq Q(s).
    \end{align}
\end{proposition}

\noindent This has the following immediate corollary. 

\begin{corollary}
For any regular values $0 \le s \le \tau < \infty$ of $w$ we have $Q_*(\tau) \ge Q_*(s)$ and $Q^*(\tau) \ge Q^*(s)$. 
\end{corollary}

% Since $f_*,g_*, h_*$ all satisfy the assumptions of Proposition \ref{proposition:monotone-ODE}, if we define
% \begin{align*}
%     X_*=&4\pi (3-p)^2f_*(w)e^{-w}c_0^{-1}|\nabla w|^{p-2}\nabla w+(g_*(w)+2(3-p)h_*(w))|\nabla w|\nabla w\\
%       &+(3-p)^2h_*(w)(-\Delta w\frac{\nabla w}{|\nabla w|}+\nabla|\nabla w|),
% \end{align*}
% and run the same arguments as for $X$, we can obtain
% \begin{theorem}
%     For regular value $0\leq s\leq \tau<\infty$, we have
%     \begin{align}
%         Q_*(\tau)\geq Q_*(s).
%     \end{align}
% \end{theorem}

\section{Asymptotic behavior of $Q_*$ and $Q^{*}$}
\label{section:asymptotics}

In this section, we obtain the asymptotic behavior of the functions $Q_*$ and $Q^{*}$ as $t\to\infty$, and then complete proof of Corollary \ref{main-theorem}. To begin, recall that the asymptotic volume ratio of an end $E$ is defined as
\begin{align*}
    \text{Avg}(g;E)=\lim_{R\to\infty}\frac{B_R(o)\cap E}{\omega_3R^3},
\end{align*}
where $\omega_3$ is the volume of unit ball in $\mathbb{R}^3$, and $o\in M$ is any fixed point. Since $(M^3,g)$ is asymptotically flat with a single end $E$, we have $\text{Avg}(E;g)=1$.

In \cite{benatti2022asymptotic}, Benatti, Fogagnolo and Mazzieri proved that for complete $\mathcal C^{1,\alpha}$ asymptotically
conical manifolds, the solution $u$ to (\ref{u-pde}) has the following expansion:
\begin{align*}
    u=\frac{p-1}{3-p}\frac{c_p}{\text{Avg}(g;E)^{\frac{1}{p-1}}r^{\frac{3-p}{p-1}}}+o_2(r^{-\frac{3-p}{p-1}}).
\end{align*}
Since the asymptotically flat manifold $(M,g)$ is complete $\mathcal C^{1,\alpha}$ asymptotically conical, and its
asymptotic volume ratio is $1$, we obtain
\begin{align}
    \label{general-u-expansion}
    u=\frac{p-1}{3-p}\frac{c_p}{r^{\frac{3-p}{p-1}}}+o_2(r^{-\frac{3-p}{p-1}}).
\end{align}
First we give the asymptotic behavior of $Q_*(t)$. 

\begin{proposition}
One has $\lim_{t\to\infty} Q_*(t) = 0$. 
\end{proposition}

\begin{proof}
    Recall that the functions $f_*(t)$, $g_*(t)$, and $h_*(t)$ all go to 0 as $t\to \infty$. It follows easily from (\ref{general-u-expansion}) that $W(t) \to 4\pi(3-p)^2$ and $\frac{dW}{dt}(t) \to 0$ as $t\to \infty$. Therefore 
    \[
    Q_*(t) = 4\pi(3-p)^2 f_*(t) + g_*(t)W(t) + (p-1)(3-p)h_*(t)\frac{dW}{dt}(t) \to 0
    \]
    as $t\to \infty$. This proves the proposition. 
\end{proof}

Now, we give the asymptotic behavior of $Q^*(t)$.

\begin{proposition}
    One has
    \begin{align}
        \lim_{t\to\infty}Q^*(t)\leq 8\pi(3-p)^2\tilde{c}_{p,s}\left(\frac{p-1}{3-p}c_p\right)^{-\frac{p-1}{3-p}} m_{\text{ADM}}-4\pi\left((3-p)^2+4\right).
    \end{align}
\end{proposition}

\begin{proof}
    We rewrite $Q^{*}(t)$ into two parts:
    \begin{align*}
        Q^{*}(t)=A^*(t)+B^*(t),
    \end{align*}
    where
    \begin{align*}
        A^*(t)&=4\pi(3-p)^2f^{*}-(3-p)^2h^{*}\int_{\Sigma_t}|\nabla w|H+(3-p)h^{*}\int_{\Sigma_t}|\nabla w|^2,\\
        B^*(t)&=\left(g^{*}+(3-p)h^{*}\right)\int_{\Sigma_t}|\nabla w|^2.
    \end{align*}
In \cite{agostiniani2022riemannian}, they defined 
\begin{align*}
    F_p(t)=\frac{1}{3-p}\left(\frac{p-1}{3-p}c_p\right)^{\frac{p-1}{3-p}}e^{\frac{t}{3-p}}\left[4\pi(3-p)-\int_{\Sigma_t}|\nabla w|H+\frac{1}{3-p}\int_{\Sigma_t}|\nabla w|^2\right],
\end{align*}
and proved using (\ref{general-u-expansion}) that
\begin{align*}
    \lim_{t\to\infty}F_p(t)\leq 8\pi m_{\text{ADM}}.
\end{align*}
By Remark \ref{remark: asy behavior of coefficient functions}, we have
\begin{gather*}
    h^{*}\sim \frac{\tilde{c}_{p,s}}{3-p}e^{\frac{t}{3-p}}+O(e^{-\frac{t}{3-p}}),\\
    \lim_{t\to\infty}\frac{f^{*}}{h^{*}}=3-p.
\end{gather*}
Thus we obtain
\begin{align*}
    \lim_{t\to\infty}\frac{A^*(t)}{F_p(t)}=(3-p)^2\tilde{c}_{p,s}\left(\frac{p-1}{3-p}c_p\right)^{-\frac{p-1}{3-p}},
\end{align*}
and it then follows that
\begin{align}
    \lim_{t\to\infty}A^*(t)\leq 8\pi(3-p)^2\tilde{c}_{p,s}\left(\frac{p-1}{3-p}c_p\right)^{-\frac{p-1}{3-p}} m_{\text{ADM}}.
\end{align}
Also, we have
\begin{align*}
    g^{*}+(3-p)h^{*}\sim -(3-p)^2-4+O(r^{-1}),
\end{align*}
which shows that
\begin{align}
    \lim_{t\to\infty}B^*(t)=-4\pi\left((3-p)^2+4\right).
\end{align}
Therefore, we obtain
\begin{align*}
    \lim_{t\to\infty}Q^*(t)\leq 8\pi(3-p)^2\tilde{c}_{p,s}(\frac{p-1}{3-p}c_p)^{-\frac{p-1}{3-p}} m_{\text {ADM}}-4\pi\left((3-p)^2+4\right),
\end{align*}
as needed.
\end{proof}

We can now give the proof of Corollary \ref{main-theorem}. 

\begin{proof}(Corollary \ref{main-theorem}). 
Let $(M^3,g)$ be a complete, asymptotically flat manifold with non-negative scalar curvature and minimal boundary $\Sigma$. Assume that $H_2(M,\Sigma)=0$. The quantity $Q_*$ is non-decreasing so we obtain 
\begin{align*}
0 = \lim_{t\to \infty} Q_*(t) \ge Q_*(0) &= 
4\pi(3-p)^2 f_*(0) + g_*(0)W(0) + (p-1)(3-p)h_*(0) \frac{dW}{dt}(0) ,\\
&= 4\pi(3-p)^2 f_*(0) + \bigg[g_*(0)+2(3-p)h_*(0)\bigg]W(0),
\end{align*}
where equation (\ref{equation:dwdt}) is used to get the last line. 
Proposition \ref{proposition:decaying-value-0} yields $
f_*(0)<0$ and $g_*(0)+2(3-p)h_*(0) >0$.
Thus we obtain
\begin{align*}
W(0) \le \frac{-4\pi (3-p)^2 f_*(0)}{g_*(0)+2(3-p)h_*(0)}. 
\end{align*}
According to Proposition \ref{proposition:decaying-value-0} again, the right hand side of the above inequality is equal to $W_s(0)$,
where $W_s$ is the $W$ function on mass 2 Schwarzschild. To summarize, we have now shown that
\begin{equation}
\label{eq:initial}
W(0)\le W_s(0)
\end{equation}
by exploiting the monotonicity of $Q_*$.

The quantity $Q^*$ is also  non-decreasing and so 
\[
Q^*(0) \le \lim_{t\to\infty} Q^*(t) \le 8\pi(3-p)^2\tilde{c}_{p,s}(\frac{p-1}{3-p}c_p)^{-\frac{p-1}{3-p}} m_{ADM}-4\pi\left((3-p)^2+4\right).
\]
On the other hand,  let $Q^*_s$ be the $Q^*$ function on mass 2 Schwarzschild. According to Proposition \ref{proposition:growing-value-0}, we have $g^*(0)+2(3-p)h^*(0) < 0$. Together with equations (\ref{equation:dwdt}) and (\ref{eq:initial}), it follows that  
\begin{align*}
Q^*(0) &= 4\pi (3-p)^2 f^*(0) + g^*(0) W(0) + (p-1)(3-p)h^*(0)\frac{dW}{dt}(0)\\
&= 4\pi (3-p)^2 f^*(0) + \bigg[g^*(0) + 2(3-p)h^*(0)\bigg]W(0) \\
&\ge 4\pi (3-p)^2 f^*(0) + \bigg[g^*(0) + 2(3-p)h^*(0)\bigg]W_s(0)\\
&=4\pi (3-p)^2 f^*(0) + g^*(0) W_s(0) + (p-1)(3-p)h^*(0)\frac{dW_s}{dt}(0) = Q^*_s(0). 
\end{align*}
Since $Q^*_s$ is constant, we thus obtain
\begin{align*}
    Q^*(0) \ge Q^*_s(0) = \lim_{t\to \infty}Q^*_s(t) = 16\pi(3-p)^2 \left(\frac{p-1}{3-p}\right)^{-\frac{p-1}{3-p}} - 4\pi\left((3-p)^2 + 4\right).
\end{align*}
Combining these inequalities, we obtain 
\[
\tilde c_{p,s}\left(c_p\right)^{-\frac{p-1}{3-p}}m_{\text{ADM}} \ge 2.
\]
This is equivalent to 
\[
m_{\text{ADM}} \ge 2\left(\frac{C_p}{K_p}\right)^{\frac{1}{3-p}},
\]
as needed. 

Now, we prove the rigidity case.
Let $\mathcal{T}$ denote the set of regular values of $w$, and suppose $T=\sup_{L>0}\{[0,L)\subset\mathcal{T}\}$. Such $T$ exists because we assume $\partial M$ is nontrivial. 
When the equality is achieved, $Q^*$ remains constant for $t\in\mathcal{T}$, and thus for $t\in [0,T)$, we have $(Q^*)'(t)\equiv 0$, which requires
\begin{align*}
    \chi(\Sigma_t)=2, \quad 
    A_{ij}=\frac{H}{2}g_{ij}, \quad 
    \nabla^{\Sigma_t}|\nabla w|=0, \quad 
    R_g=0.
\end{align*}
From above conditions, we know the region $M_T:=\{0\leq w< T\}$ is topologically $[0,s_0)\times S^2$, and its metric can be written as
\begin{align*}
    g=ds^2+\varphi(s)^2g_0,
\end{align*}
where $g_0$ is the round metric of the unit sphere $S^2$. By the same argument as in Proposition \ref{proposition:W-monotonicity}, the metric must be given by
\begin{align*}
    g=ds^2+\psi(s)^2g_0, \quad \forall x\in M_T,
\end{align*}
where $([0,\infty)\times S^2,ds^2+\psi(s)^2g_0)$ is the Schwarzschild metric, whose horizon has the same area as $\partial M$.

Now suppose for contradiction that $T < \infty$. Let $u_s$ be the $p$-harmonic function on Schwarzschild with $u_s = 1$ on $\{0\}\times S^2$ and $u_s\to 0$ at infinity. Then necessarily $1-u$ is a positive multiple of $1-u_s$ on $M_T$.  Indeed, both functions are $p$-harmonic on $M_T$, take the value 0 on $\{0\}\times S^2$, and take constant values on the other boundary component of $M_T$. Hence the maximum principle implies that $1-u$ is a multiple of $1-u_s$. Since the gradient of $u_s$ never vanishes, this implies that $\vert \grad u\vert \neq 0$ on the level set $\{w=T\}$, and hence there exists $\eps > 0$ such that $[0,T+\eps)\subset \mathcal T$, contradicting the definition of $T$. 

Therefore, $w$ has no critical values, and $Q^*$ is identically constant for all $t\in [0,\infty)$, which implies \eqref{W-ODE} achieves equality. By Proposition \ref{proposition:W-monotonicity}, $(M,g)$ is isometric to spatial Schwarzschild with horizon boundary.
\end{proof}

% Then $u$ solves $p$-harmonic function on $M_T$ with Dirichlet boundary condition, and we can extend $u$ to $[0,\infty)\times S^2,ds^2+\psi(s)^2g_0$, denoted as $\tilde{u}$, where $\tilde{u}$ solves $p$-harmonic function on the entire Schwarzschild, satisfying $\tilde{u}|_{\{0\}\times S^2}=1$, $\lim_{x\to\infty}\tilde{u}=c$.
% Here, $c\neq 1$ is a constant. $\tilde{u}$ has no critical points on the entire manifold, which in turn implies that $|\nabla w|\neq 0$ on the level set $\{w=T\}$, and hence there exists $\epsilon>0$, such that $[0,T+\epsilon)\subset\mathcal{T}$, contradicted to the definition of $T$.

% Therefore, $w$ has no critical value, and $Q^*\equiv constant$ for all $t\in [0,\infty)$, which implies \eqref{W-ODE} should achieve equality, and by Proposition \ref{proposition:W-monotonicity}, $(M,g)$ is isometric to spatial Schwarzschild with horizon boundary.

\bibliographystyle{mrl}
\bibliography{ref.bib}

\end{document}